\documentclass{m2an}
\usepackage{amsmath}
\usepackage{amssymb} 
\usepackage{mathptmx}      
\usepackage{amsfonts}
\usepackage{bm} 
\usepackage{booktabs} 
\usepackage{multirow} 
\usepackage{graphicx}
\usepackage{mathtools} 
\usepackage{tikz}
\usepackage{stmaryrd} 
\usepackage{subcaption} 
\usepackage{url} 
\allowdisplaybreaks

\newcommand{\vertiii}[1]{{\left\vert\kern-0.25ex\left\vert\kern-0.25ex\left\vert #1
        \right\vert\kern-0.25ex\right\vert\kern-0.25ex\right\vert}}

\makeatletter
\newcommand{\fcolon}{%
  \mathrel{\mathpalette\fcolon@\relax}%
}
\newcommand{\fcolon@}[2]{%
  \sbox\z@{$\m@th#1:$}%
  \vbox to\ht\z@{%
    \hbox{$\m@th#1.$}%
    \vss
    \hbox{$\m@th#1.$}%
    \vss
    \hbox{$\m@th#1.$}%
  }%
}
\makeatother

\newcommand{\Ivec}{\mathbf{I}}
\newcommand{\Qvec}{\mathbf{Q}}
\newcommand{\Gvec}{\mathbf{G}}
\newcommand{\Nvec}{\mathbf{N}}
\newcommand{\Pvec}{\mathbf{P}}

\newcommand*{\llbrace}{\left\{\mskip-5mu\left\{}
\newcommand*{\rrbrace}{\right\}\mskip-5mu\right\}}

\usepackage[colorlinks=true,linkcolor=black,anchorcolor=black,citecolor=black,filecolor=black,menucolor=black,runcolor=black,urlcolor=black]{hyperref}
\usepackage{cleveref} 
\usepackage{diagbox} 

\theoremstyle{plain}
\newtheorem{theorem}{Theorem}[section]
\newtheorem{lemma}[theorem]{Lemma}

\newtheorem{remark}[theorem]{Remark}

\usepackage{graphicx}
\graphicspath{{Pictures/}} 

\usepackage{amsopn}

\begin{document}
\title{Variational and numerical analysis of a $\Qvec$-tensor model for smectic-A liquid crystals}
\thanks{The work of JX is supported by the National Natural Science Foundation of China (No. 12201636) and the Research Fund of National University of Defense Technology [grant number ZK22-37]. The work of PEF is supported by EPSRC [grant numbers EP/R029423/1 and EP/W026163/1].}


\author{Jingmin Xia}\address{College of Meteorology and Oceanography, National University of Defense Technology, Changsha, China;
\email{jingmin.xia@nudt.edu.cn}}
\author{Patrick E.~Farrell}\address{Mathematical Institute, University of Oxford, Oxford, UK; \email{patrick.farrell@maths.ox.ac.uk}}

\date{The dates will be set by the publisher}

\begin{abstract}
    We analyse an energy minimisation problem recently proposed for modelling smectic-A liquid crystals.
    The optimality conditions give a coupled nonlinear system of partial differential equations, with a second-order equation for the tensor-valued nematic order parameter $\Qvec$ and a fourth-order equation for the scalar-valued smectic density variation $u$.
    Our two main results are a proof of the existence of solutions to the minimisation problem, and the derivation of {a priori} error estimates for its discretisation of the decoupled case (i.e., $q=0$) using the $\mathcal{C}^0$ interior penalty method.
    More specifically, optimal rates in the $H^1$ and $L^2$ norms are obtained for $\Qvec$, while optimal rates in a mesh-dependent norm and $L^2$ norm are obtained for $u$.
    Numerical experiments confirm the rates of convergence.
\end{abstract}

%

\maketitle
\section*{Introduction}

Smectic liquid crystals (LC) are layered mesophases that have a periodic modulation of the mass density along one spatial direction.
Informally, they can be thought of as one-dimensional solids along the direction of periodicity and two-dimensional fluids along the other two remaining directions.
According to different in-layer structures, several types of smectic LC are recognised, e.g., smectic-A, smectic-B and smectic-C (see \cite[Figure 9]{sackmann-1989-article} for illustrations of different smectic phases).
In particular, in the smectic-A phase, the molecules tend to align parallel to the normals of the smectic layers.
For a broader review of liquid crystals, see \cite{ball-2017-article, gennes-book, stewart-2004-book}.

Several models have been proposed to describe smectic-A liquid crystals. The classical de Gennes model \cite{gennes-1972-article} employs a complex-valued order parameter to describe the magnitude and phase of the density variation. This complex-valued parameterisation leads to some key modelling difficulties, which motivated the development by Pevnyi, Selinger \& Sluckin (PSS) of a smectic-A model with a real-valued smectic order parameter and a vector-valued nematic order parameter \cite{pevnyi-2014-article}. The use of a vector-valued nematic order parameter limits the kinds of defects the model can permit~\cite{ball-2017-article}, and hence Ball \& Bedford (BB) proposed a version of the PSS model employing a tensor-valued nematic order parameter \cite{ballbed-2015-article} instead.
The model considered in this work is similar to the BB model, but with additional modifications to render it amenable to numerical simulation (see~\cite{xia-2021-article} for details). The model was used to numerically simulate several key smectic defect structures, such as oily streaks and focal conic domains, for the first time.

While the numerical modelling of nematic liquid crystals is now mature, relatively little work has considered smectics. Garc\`{i}a-Cervera and Joo~\cite{garcia-cervera-2010} perform numerical simulations of the classical de Gennes model~\cite{gennes-1972-article} in the presence of a magnetic field, using a combined Fourier-finite difference approach. Wittmann et al.~use density functional theory to investigate the topology of smectic liquid crystals in complex confinement~\cite{rene-2021-article}. Monderkamp et al.~examine the topology of defects in two-dimensional confined smectics with the help of extensive Monte Carlo simulations~\cite{paul-2021-article}. Our goal in this work is to analyse a model for smectic-A liquid crystals, and its finite element discretization, that was recently proposed by Xia, MacLachlan, Atherton and Farrell in~\cite{xia-2021-article}.

We consider an open, bounded and convex spatial domain $\Omega\subset \mathbb{R}^d$, $d\in \{2,3\}$ with Lipschitz boundary $\partial\Omega$. The smectic-A free energy we analyse in this work is given by
\begin{equation}
    \label{eq:unify}
    \mathcal{J}(u, \Qvec) = \int_\Omega \left(f_s(u) + B\left|\mathcal{D}^2 u + q^2 \left(\Qvec+\frac{\Ivec_d}{d}\right) u \right|^2 +f_n(\Qvec,\nabla \Qvec) \right),
\end{equation}
where
$\Ivec_d$ denotes the $d\times d$ identity matrix,
$\mathcal{D}^2$ is the Hessian operator, $f_s(u)$ is the smectic bulk energy density
\begin{equation*}
    f_s(u) \coloneqq \frac{a_1}{2} u^2 +\frac{a_2}{3} u^3+\frac{a_3}{4} u^4
\end{equation*}
and $f_n(\Qvec,\nabla \Qvec)$ is the nematic bulk energy density
\begin{equation}
    \label{eq:f_n}
    \begin{aligned}
        f_n(\Qvec,\nabla \Qvec) &= f_n^e(\nabla \Qvec) + f_n^b(\Qvec)\\
                        &\coloneqq \frac{K}{2}|\nabla \Qvec|^2
                        +\begin{cases}
                            \left(-l \left(\text{tr}(\Qvec^2)\right) + l \left(\text{tr}(\Qvec^2)\right)^2\right), &\text{if }d=2,\\
                            \left(-\frac{l}{2} \left(\text{tr}(\Qvec^2)\right) - \frac{l}{3} \left(\text{tr}(\Qvec^3)\right) + \frac{l}{2} \left(\text{tr}(\Qvec^2)\right)^2\right), & \text{if }d=3.
                        \end{cases}
    \end{aligned}
\end{equation}
Here, $B > 0$ is the nematic-smectic coupling parameter, $a_1,a_2,a_3$ represent smectic bulk constants with $a_3 > 0$, $K > 0$ and $l > 0$ are nematic elastic and bulk constants, respectively, $q \ge 0$ is the wave number of the smectic layers, and the trace of a matrix is given by $\mathrm{tr}(\cdot)$.

The two main contributions of this paper are to prove existence of minimisers to the problem of minimising $\mathcal{J}$ over a suitably-defined admissible set, and to derive
{a priori} error estimates for its discretisation using the $\mathcal{C}^0$ interior penalty method~\cite{brenner-2011-book,brenner-2005-article}.
We show the existence result of minimising the free energy \cref{eq:unify} in \cref{sec:existence}.
We then derive a priori error estimates for both $\Qvec$ and $u$ in the simplified case $q = 0$ in \cref{sec:decoupled}.
We confirm that the theoretical predictions match numerical experiments in \cref{sec:num}, for both $q = 0$ and $q > 0$.


In order to avoid the proliferation of constants throughout this work, we use the notation $a \lesssim b$ (respectively, $b \gtrsim a$) to represent the relation $a \le C b$ (respectively, $b \ge C a$) for some generic constant $C$ (possibly not the same constant on each appearance) independent of the mesh.

\section{Existence of minimisers}
\label{sec:existence}

To formulate the minimisation problem for \cref{eq:unify}, we must first define the admissible space $\mathcal{A}$
in which minimisers will be sought. We define $\mathcal{A}$ as
\begin{equation}
    \label{eq:admis}
        \mathcal{A} = \bigg\{ (u, \Qvec) \in H^2(\Omega,\mathbb{R}) \times H^1(\Omega,S_0):
                                          \Qvec=\Qvec_b, u=u_b\ \text{on}\ \partial\Omega\bigg\},
\end{equation}
with the specified Dirichlet boundary data $\Qvec_b\in H^{1/2}(\partial\Omega,S_0)$ and $u_b\in H^{3/2}(\partial\Omega,\mathbb{R})$.
Here, 
$S_0$ denotes the space of symmetric and traceless $d\times d$ real-valued matrices.
For simplicity of the analysis, we only consider Dirichlet boundary conditions for $\Qvec$ and $u$ here, but other types of boundary conditions (e.g., a mixture of the Dirichlet and natural boundary conditions for $u$ as illustrated in the implementations in \cite{xia-2021-article}) can be taken.
With this admissible space, we consider the minimisation problem
\begin{equation} \label{eq:problem}
\min_{(u, \Qvec) \in \mathcal{A}} \mathcal{J}(u, \Qvec).
\end{equation}

Notice that $f_n(\Qvec,\nabla \Qvec)$ is the classical Landau de Gennes (LdG) free energy density \cite{gennes-book, mottram-2014-article} for nematic LC and it is proven by Davis \& Gartland \cite[Corollary 4.4]{davis-1998-article} that there exists a minimiser of the functional $\int_\Omega f_n(\Qvec,\nabla\Qvec)$ for $\Qvec\in H^1(\Omega,S_0)$ in three dimensions.
Moreover, Bedford \cite[Theorem 5.18]{bedford-2014-phd} gives an existence result for the BB model in three dimensions:
\begin{equation*}
    \min_{(u,\Qvec)\in \mathcal{A}^{BB}}\mathcal{J}^{BB}(u,\Qvec) = \int_\Omega \left\{\frac{K}{2}|\nabla \Qvec|^2 + B\left|\mathcal{D}^2u +q^2\left(\frac{\Qvec}{s}+\frac{\Ivec_3}{3}\right)u \right|^2 + \frac{a_1}{2} u^2 + \frac{a_2}{3} u^3 + \frac{a_3}{4} u^4\right\},
\end{equation*}
with an admissible space $\mathcal{A}^{BB}\coloneqq \{\Qvec\in SBV(\Omega,S_0),u\in H^2(\Omega,\mathbb{R}): \Qvec=s\left(\mathbf{n}\otimes\mathbf{n} -\frac{\Ivec_d}{d}\right), s\in [0,1],|\mathbf{n}|=1\}$,
where $SBV$ denotes special functions of bounded variation. For simplicity, we have ignored the surface integral here in the energy functional of the BB model.
One can observe its resemblance to \cref{eq:unify}.
Motivated by the above results, we prove the existence of minimisers to \cref{eq:unify} via the direct method of calculus of variations. 

\begin{theorem}
    \label{thm:existence}
    (Existence of minimisers)
    Let $\mathcal{J}$ be of the form \cref{eq:unify} with positive parameters $a_3$, $B$, $K$, $l$ and non-negative wave number $q$.
    Then there exists a solution pair $(u^\star,\Qvec^\star)$ that solves \cref{eq:problem}.
\end{theorem}
\begin{proof}
    Note that both the smectic density $f_s$ and the nematic bulk density $f_n^b$ are bounded from below as $a_3,l>0$.
    Thus, $\mathcal{J}$ is also bounded from below and we can choose a minimising sequence $\{(u_j,\Qvec_j)\}$, i.e.,
    \begin{equation}
        \label{eq:upper}
        \begin{aligned}
            (u_j,\Qvec_j)\in \mathcal{A}, \ u_j-\tilde{u}\in H^2\cap H^1_0(\Omega,\mathbb{R}), \ &\Qvec_j-\tilde{\Qvec} \in H^1_0(\Omega,S_0), \\
            \mathcal{J}(u_j,\Qvec_j) \overset{j\to \infty}{\longrightarrow} \inf\{\mathcal{J}(u,\Qvec): (u,\Qvec)\in \mathcal{A},\ u-\tilde{u}\in H^2\cap H^1_0(\Omega,\mathbb{R}),\ &\Qvec-\tilde{\Qvec}\in H^1_0(\Omega,S_0)\} < \infty.
        \end{aligned}
    \end{equation}
    Here we set $\tilde{\Qvec}\in H^1(\Omega,S_0)$ (resp.\ $\tilde{u}\in H^2(\Omega,\mathbb{R})$) to be any function with trace $\Qvec_b$ (resp.\ $u_b$).
    We tackle the three terms in $\mathcal{J}(u,\Qvec)$ separately in the following.

    First, for the nematic energy term $\int_\Omega f_n(\Qvec,\nabla \Qvec)$, we can follow the proof of \cite[Theorem 4.3]{davis-1998-article} to obtain that $f_n(\Qvec_j,\nabla \Qvec_j)$ is coercive in $H^1(\Omega,S_0)$, i.e., $f_n$ grows unbounded as $\|\Qvec_j\|_1\rightarrow \infty$, and thus the minimising sequence $\{\Qvec_j\}$ must be bounded in $\mathbf{H}^1(\Omega,S_0)$.
    Since $H^1(\Omega)$ is a reflexive Banach Space, we have a subsequence (also denoted as $\{\Qvec_j\}$) that weakly converges to $\Qvec^\star\in H^1(\Omega,S_0)$ such that $\Qvec^\star-\tilde{\Qvec} \in H^1_0(\Omega,S_0)$, and from the Rellich--Kondrachov theorem it follows that
    $
    \Qvec_j \rightarrow \Qvec^\star \text{ in } L^2(\Omega) \text{ and }\nabla \Qvec_j \rightharpoonup \nabla \Qvec^\star \text{ in } L^2(\Omega).
    $
    The weak lower semi-continuity of the nematic energy density $f_n$ in \cref{eq:f_n} is guaranteed by \cite[Lemma 4.2]{davis-1998-article}, therefore,
    \begin{equation}
        \label{eq:lim1}
        \liminf_{j\to\infty} \int_{\Omega} f_n(\Qvec_j,\nabla \Qvec_j) \ge \int_\Omega f_n(\Qvec^\star,\nabla \Qvec^\star).
    \end{equation}

    For the smectic bulk term $\int_\Omega f_s(u)$, we can follow the proof in \cite[Theorem 5.19]{bedford-2014-phd}.
    By \cref{eq:upper}, we have
    $$
        \sup_j \int_\Omega \left( \left|\mathcal{D}^2 u_j\right|^2+ |u_j|^2\right) < \infty,
    $$
    which implies an upper bound for $\nabla u_j$ using
        \cite[Ineq.\ (5.42)]{bedford-2014-phd}.
        Hence, $\{u_j\}$ is bounded in $H^2(\Omega)$ and thus there is a subsequence (also denoted as $\{u_j\}$) such that
    $
    u_j \rightharpoonup u^\star \text{ in } H^2(\Omega)
    $
    and $u^\star-\tilde{u}\in H^2\cap H^1_0(\Omega)$.
    Moreover, one can readily check that $\|u^\star\|_\infty < \infty$ by the Sobolev embedding of $H^2(\Omega)$ into the H\"older spaces $\mathcal{C}^{\mathfrak{t},\varkappa_0}(\Omega)$ ($\mathfrak{t}+\varkappa_0=1$ for $d=2$ and $\mathfrak{t}+\varkappa_0=1/2$ for $d=3$) and the boundedness of domain $\Omega$.
    Again, it follows from the Rellich--Kondrachov theorem that
    $
    u_j \rightarrow u^\star \text{ in }L^2(\Omega) \text{ and } \mathcal{D}^2 u_j \rightharpoonup \mathcal{D}^2 u^\star \text{ in } L^2(\Omega).
    $
    Noting $f_s(u)$ is bounded from below for all $u\in H^2(\Omega)$, we then obtain
    \begin{equation}
        \label{eq:lim2}
        \liminf_{j\to\infty} \int_\Omega f_s(u_j) \ge \int_\Omega f_s(u^\star).
    \end{equation}

    Now, we consider the nematic-smectic coupling term $\int_\Omega B\left|\mathcal{D}^2 u + q^2 \left(\Qvec+\frac{\Ivec_d}{d}\right) u \right|^2$ in $\mathcal{J}(u,\Qvec)$.
    Note that when the wave number $q=0$, this term reduces to $\int_\Omega B |\mathcal{D}^2 u|^2$ and it is straightforward to obtain the weak lower semi-continuity property.
    Therefore, we discuss the case of $q>0$ in detail as follows.
    By the $\mathbf{H}^1$-boundedness property of the minimising sequence $\{\Qvec_j\}$ and the fact that $\|u^\star\|_\infty<\infty$, we can deduce
        \begin{align*}
            \int_\Omega |u_j \Qvec_j -u^\star \Qvec^\star|^2 &= \int_\Omega \left| (u_j-u^\star)\Qvec_j + u^\star(\Qvec_j-\Qvec^\star)\right|^2 \\
                                                     & \le 2\int_\Omega \left(|u_j-u^\star|^2|\Qvec_j|^2 + |u^\star|^2|\Qvec_j-\Qvec^\star|^2\right) \\
                                             & \to 0 \quad \text{as }u_j\to u^\star, \Qvec_j\to \Qvec^\star \text{ in } L^2.
        \end{align*}
        Hence, $u_j\Qvec_j\to u^\star \Qvec^\star$ in $L^2(\Omega)$, and further,
    \begin{equation*}
        \begin{aligned}
            u_j\left(\Qvec_j+\frac{\Ivec_d}{d}\right) &\to u^\star \left(\Qvec^\star+\frac{\Ivec_d}{d}\right) && \quad \text{in } L^2(\Omega), \\
            u_j \left(\Qvec_j+\frac{\Ivec_d}{d}\right) \colon \mathcal{D}^2 u_j &\rightharpoonup u^\star\left(\Qvec^\star+\frac{\Ivec_d}{d}\right) \colon \mathcal{D}^2 u^\star && \quad \text{in } L^1(\Omega).
    \end{aligned}
    \end{equation*}
    Therefore, we have
    \begin{align}
             \liminf_{j\to \infty} \int_\Omega \left|\mathcal{D}^2 u_j + q^2 \left(\Qvec_j+\frac{\Ivec_d}{d}\right) u_j\right|^2 
            & =\liminf_{j\to \infty} \int_\Omega \left( \left|\mathcal{D}^2 u_j\right|^2 + 2q^2 u_j \left(\Qvec_j+\frac{\Ivec_d}{d}\right)\colon \mathcal{D}^2 u_j + q^4\left|u_j\left(\Qvec_j+\frac{\Ivec_d}{d}\right)\right|^2 \right) \nonumber\\
            & \ge \int_\Omega \left( \left|\mathcal{D}^2 u^\star\right|^2 + 2q^2 u^\star \left(\Qvec^\star+\frac{\Ivec_d}{d}\right)\colon \mathcal{D}^2 u^\star + q^4\left|u^\star\left(\Qvec^\star+\frac{\Ivec_d}{d}\right)\right|^2 \right) \nonumber\\
            & = \int_\Omega \left|\mathcal{D}^2 u^\star + q^2 \left(\Qvec^\star+\frac{\Ivec_d}{d}\right) u^\star\right|^2.
        \label{eq:lim3}
    \end{align}

    Hence, we can conclude that $\mathcal{J}(u^\star, \Qvec^\star)$ achieves its minimum in the admissible space $\mathcal{A}$ by combining \cref{eq:lim1}, \cref{eq:lim2} and \cref{eq:lim3}.
\end{proof}

\begin{remark}
We briefly compare the proof with that of Bedford \cite[Theorem 5.18]{bedford-2014-phd}. In that work, the admissible space $\mathcal{A}^{BB}$ included a uniaxial constraint.
This assumption is necessary to guarantee the $\mathbf{H}^1$-boundedness property of the minimising sequence $\{\Qvec_j\}$, since that work seeks $\Qvec \in SBV(\Omega,S_0)$ instead of $\mathbf{H}^1$. Enforcing this constraint numerically is difficult~\cite{borthagaray2020}; in this work we prefer the choice $\Qvec \in \mathbf{H}^1$, which enables us to remove the uniaxiality assumption.
\end{remark}

\section{A priori error estimates}
\label{sec:decoupled}

We now consider the discretisation of the minimisation problem \cref{eq:problem}.
For simplicity, we analyse the decoupled case with $q=0$, where \cref{eq:problem} splits into two independent problems:
one for the smectic density variation $u$:
\begin{equation*}
    \underset{u\in H^2\cap H^1_b(\Omega,\mathbb{R})}{\min}\quad \mathcal{J}_1(u) = \int_\Omega \left(B\left|\mathcal{D}^2 u\right|^2 + f_s(u) \right),
\end{equation*}
and one for the tensor field $\Qvec$,
\begin{equation*}
    \underset{\Qvec\in H^1_b(\Omega,S_0)}{\min}\quad \mathcal{J}_2(\Qvec) = \int_\Omega \left( f_n(\Qvec,\nabla \Qvec)\right).
\end{equation*}
Here,
$H^1_b(\Omega,\mathbb{R}) \coloneqq \{u\in H^1(\Omega,\mathbb{R}): u=u_b\text{ on }\partial\Omega\}$
and
$H^1_b(\Omega,S_0) \coloneqq \{\Qvec\in H^1(\Omega,S_0): \Qvec=\Qvec_b\text{ on }\partial\Omega\}$. 
One can derive the following strong forms of their equilibrium equations using integration by parts.
The optimality condition for $u$ yields a fourth-order partial differential equation (PDE)
\begin{equation*}
    \begin{cases}
        2B \nabla \cdot \left(\nabla\cdot\mathcal{D}^2 u\right) +a_1 u + a_2 u^2 +a_3 u^3 = 0 &\text{ in }\Omega,\\
        u=u_b,\quad 
        \mathcal{D}^2 u\cdot \nu=\mathcal{D}^2 u_b\cdot\nu &\text{ on }\partial\Omega,
    \end{cases}
\end{equation*}
where $\nu$ denotes the outward unit normal.
Note that the second boundary condition of $u$ is a natural boundary condition on the second derivative of $u$.
For simplicity, we only consider the case of a cubic nonlinearity (i.e., $a_2=0$) here as the quadratic term can be tackled similarly.
Therefore, we consider the following
governing equations
\begin{equation}
    \label{eq:strong-u}
    (\mathcal{P}1)\quad
    \begin{cases}
        2B \nabla\cdot\left(\nabla\cdot \mathcal{D}^2 u\right) +a_1 u + a_3 u^3 = 0 &\text{in }\Omega,\\
        u= u_b, \quad 
        \mathcal{D}^2 u\cdot \nu=\mathcal{D}^2 u_b\cdot\nu &\text{on }\partial\Omega.
    \end{cases}
\end{equation}

Meanwhile, the optimality condition for $\Qvec$ yields a second-order PDE
\begin{equation}
    \label{eq:strong-Q}
    (\mathcal{P}2) \quad
    \begin{cases}
        d=2\Rightarrow
            - K\Delta \Qvec + 2l\left(2|\Qvec|^2-1\right)\Qvec = 0 &\text{ in }\Omega, \\
        d=3\Rightarrow
            - K\Delta \Qvec + l\left(-\Qvec-|\Qvec|^2+2|\Qvec|^2\Qvec \right) = 0 &\text{ in }\Omega, \\
        \Qvec = \Qvec_b &\text{ on }\partial\Omega,
    \end{cases}
\end{equation}
We now consider these two problems $(\mathcal{P}1)$ and $(\mathcal{P}2)$ in turn.

\begin{remark}
    The uniqueness of solutions is not expected. It is well-known that \cref{eq:strong-Q} can support multiple solutions~\cite{robinson2017}, while numerical experiments in \cite{xia-2021-article} indicate the existence of multiple solutions to the optimality conditions for \cref{eq:problem}.
\end{remark}

\subsection{A priori error estimates for $(\mathcal{P}1)$}
\label{sec:apriori-u}

Since the PDE \cref{eq:strong-u} for the density variation $u$ is a fourth-order problem, a conforming discretisation requires a finite dimensional subspace of the Sobolev space $H^2(\Omega)$, which necessitates the use of $\mathcal{C}^1$-continuous elements.
The construction of these elements is quite involved, particularly in three dimensions; without a special mesh structure, the lowest-degree conforming
elements are the Argyris \cite{argyris-1968-article} and Zhang \cite{zhang-2009-article} elements, of degree 5 and 9 in two and three dimensions respectively.
One approach to avoid such complexity is to use mixed formulations by solving two second-order systems, and we refer to \cite{cheng-2000-article, scholtz-1978-article} for instance.
However, this substantially increases the size of the linear systems to be solved.
Alternatively, one can directly tackle the fourth-order problem with non-conforming elements, that do not satisfy the $\mathcal{C}^1$-requirement.
For instance, the so-called \emph{continuous/discontinuous Galerkin} methods and \emph{$\mathcal{C}^0$ interior penalty} methods ($\mathcal{C}^0$-IP) are analysed in \cite{brenner-2005-article, engel-2002-article}, combining concepts from the theory of continuous and discontinuous Galerkin methods.
Essentially, these methods use $\mathcal{C}^0$-conforming elements and penalise inter-element jumps in first derivatives to weakly enforce $\mathcal{C}^1$-continuity.
This has the advantages of both convenience and efficiency: the weak form is simple, with only minor modifications
from a conforming method, and fewer degrees of freedom are used than with a fully discontinuous Galerkin method.

We thus adopt the idea of $\mathcal{C}^0$-IP methods to solve the nonlinear fourth-order problem $(\mathcal{P}1)$ and derive a priori error estimates regarding $u$.
We adapt the techniques of \cite{maity-2020a-article} to prove convergence rates with the use of familiar continuous Lagrange elements for the problem $(\mathcal{P}1)$.
The weak form of \cref{eq:strong-u} is defined as: find $u\in H^2(\Omega)\cap H^1_b(\Omega; \mathbb{R})$ such that
\begin{equation}
    \label{eq:weak-uorig}
    \mathcal{N}^s(u)t \coloneqq A^s(u,t) + B^s(u,u,u,t) + C^s(u,t)= L^s(t) \quad \forall t\in H^2(\Omega)\cap H^1_0(\Omega),
\end{equation}
where for $t,w\in H^2(\Omega)$,
    \begin{equation*}
        A^s(t,w) = 2B \int_\Omega \mathcal{D}^2 t\colon \mathcal{D}^2 w,\
        C^s(t,w) = a_1 \int_\Omega tw,\
        L^s(t) \coloneqq 2B\int_{\partial\Omega} \left(\mathcal{D}^2 u_b \cdot\nabla t \right)\cdot\nu,
    \end{equation*}
and for $\mu,\zeta,\eta,\xi\in H^2(\Omega)$,
\begin{equation*}
    B^s(\mu, \zeta, \eta, \xi) = a_3 \int_\Omega \mu \zeta \eta \xi.
\end{equation*}
Since \cref{eq:weak-uorig} is nonlinear, we derive its linearisation: find $v\in H^2(\Omega) \cap H^1_0(\Omega)$ such that
\begin{equation}
    \label{eq:linear-uorig}
    \langle \mathcal{D}\mathcal{N}^s(u)v,w\rangle_{H^2} \coloneqq A^s(v, w) + 3B^s(u,u,v,w) + C^s(v,w) = L^s(w) \quad \forall w\in H^2(\Omega)\cap H^1_0(\Omega),
\end{equation}
where $\langle \cdot,\cdot\rangle_{H^2}$ represents the dual pairing between $\left(H^2(\Omega)\cap H^1_0(\Omega)\right)^\star$ and $H^2(\Omega)\cap H^1_0(\Omega)$.

It is straightforward to derive the coercivity and boundedness of the bilinear operator $A^s(\cdot,\cdot)$ with the semi-norm $|\cdot|_2$ (in fact, this is indeed a norm in $H^2(\Omega)\cap H^1_0(\Omega)$).

\begin{lemma}
    \label{lem:As-coer-bound}
    For $v,w\in H^2(\Omega)\cap H^1_0(\Omega)$, there holds
    $
        A^s(v,w) \lesssim |v|_2 |w|_2 \text{ and } A^s(v,v)\gtrsim |v|_2^2.
    $
\end{lemma}

Let $\mathcal{T}_h$ be a mesh of $\Omega$ with $T$ denoting an element, and let $\mathcal{E}_I$ (resp.\ $\mathcal{E}_B$) denote the set of all interior (resp.\ boundary) edges/faces $e$ of the mesh, and $\mathcal{E}\coloneqq \mathcal{E}_I\cup \mathcal{E}_B$.
Define the broken Sobolev space
$
    H^2(\mathcal{T}_h) \coloneqq \{v\in H^1(\Omega): v|_{T}\in H^2(T) \ \forall T\in \mathcal{T}_h\}
$
equipped with the broken norm $\|v\|_{2,\mathcal{T}_h}^2=\sum_{T\in \mathcal{T}_h} \|v\|^2_{2,T}$.
We take a nonconforming but continuous approximation $u_h$ for the solution $u$ of \cref{eq:weak-uorig}, that is to say, $u_h\in W_{h,b} \subset H^2(\mathcal{T}_h)\cap H^1_b(\Omega)$ defined for $\deg \ge 2$ (since $(\mathcal{P}1)$ is a fourth-order problem) by
\begin{equation*}
    \begin{aligned}
        W_{h} &\coloneqq \{ v\in H^2(\mathcal{T}_h) \cap H^1(\Omega): v\in \mathbb{Q}_{\deg}(T)\ \forall T\in \mathcal{T}_h\},\\
        W_{h,0} &\coloneqq \{ v\in H^2(\mathcal{T}_h) \cap H^1(\Omega): v=0\text{ on }\partial\Omega, v\in \mathbb{Q}_{\deg} (T)\ \forall T\in \mathcal{T}_h\},\\
    	W_{h,b} &\coloneqq \{ v\in H^2(\mathcal{T}_h) \cap H^1(\Omega): v=u_b\text{ on }\partial\Omega, v\in \mathbb{Q}_{\deg} (T)\ \forall T\in \mathcal{T}_h\},
    \end{aligned}
\end{equation*}
with $\mathbb{Q}_{\deg}$ denoting piecewise continuous polynomials of degree ${\deg}$ on a mesh of quadrilaterals or hexahedra.
Following the derivation of the $\mathcal{C}^0$-IP formulation given in \cite[Section 3]{brenner-2011-book}, we introduce the discrete nonlinear weak form: find $u_h\in W_{h,b}$ such that
\begin{equation}
    \label{eq:weak-u}
    \mathcal{N}_h^s(u_h)t_h \coloneqq A^s_h(u_h,t_h) +P^s_h(u_h,t_h)+ B^s(u_h,u_h,u_h,t_h) + C^s(u_h,t_h)
        = L^s(t_h) \quad \forall t_h\in W_{h,0},
\end{equation}
where for all $u_h,t_h\in W_h$,
\begin{equation*}
        A^s_h(u_h,t_h) \coloneqq 2B\bigg( \sum_{T\in\mathcal{T}_h} \int_T \mathcal{D}^2 u_h \colon \mathcal{D}^2 t_h
        - \sum_{e\in \mathcal{E}_I}\int_e \llbrace \frac{\partial^2 u_h}{\partial\nu^2}\rrbrace \llbracket \nabla t_h\rrbracket
    - \sum_{e\in \mathcal{E}_I} \int_e \llbrace \frac{\partial^2 t_h}{\partial\nu^2}\rrbrace \llbracket \nabla u_h\rrbracket \bigg),
\end{equation*}
and
\begin{equation}
    \label{eq:penalty}
    P^s_h(u_h,t_h) \coloneqq
    \sum_{e\in \mathcal{E}_I}\frac{2B \epsilon}{h_e^3}\int_{e} \llbracket \nabla u_h \rrbracket \llbracket \nabla t_h \rrbracket.
\end{equation}
Here, $\epsilon$ is the penalty parameter (to be specified in \cref{sec:num}),
$h_e$ indicates the size of the edge/face $e$
and the average
of the second derivatives of $u$ along the normal direction across the edge/facet $e$ is defined as
    $
        \llbrace \frac{\partial^2 u}{\partial \nu^2} \rrbrace = \frac{1}{2} \left( \left.\frac{\partial^2 u_+}{\partial \nu^2}\right|_e + \left.\frac{\partial^2 u_-}{\partial \nu^2}\right|_e \right).
    $
For any interior edge $e\in \mathcal{E}_I$ shared by two cells $T_-$ and $T_+$, we define the jump $\llbracket \mathbf{v}\rrbracket$ by $\llbracket \mathbf{v}\rrbracket = \mathbf{v}_-\cdot \nu_- + \mathbf{v}_+\cdot \nu_+$ with $\nu_-,\nu_+$ representing the restriction of outward normals in $T_-,T_+$ respectively.
On the boundary edge/face $e\in \mathcal{E}_B$, we define $\llbracket \mathbf{v}\rrbracket=\mathbf{v}\cdot\nu$.
    The operator $P^s_h$ penalises the first derivatives across the interior edge/facet since the function in $H^1(\Omega)$ is not necessarily continuously differentiable.

The linearised version is to seek $v_h\in W_{h,0}$ such that
\begin{equation}
    \label{eq:linprob-uh}
    \langle \mathcal{D}\mathcal{N}_h^s(u_h)v_h,w_h\rangle = L^s(w_h)\quad \forall w_h\in W_{h,0},
\end{equation}
where
\begin{equation}
    \label{eq:linear-uh}
    \langle \mathcal{D}\mathcal{N}_h^s(u_h)v_h,w_h\rangle
    \coloneqq A_h^s(v_h, w_h) + P_h^s(v_h,w_h) + 3B^s(u_h,u_h,v_h,w_h) + C^s(v_h,w_h).
\end{equation}
We also define the mesh-dependent $H^2$-like semi-norm for $v\in W_h$,
\begin{equation}
    \label{eq:mesh-norm}
    \vertiii{v}_h^2 \coloneqq \sum_{T\in \mathcal{T}_h} |v|_{H^2(T)}^2 + \sum_{e\in \mathcal{E}_I} \int_e \frac{1}{h_e^3} |\llbracket \nabla v \rrbracket |^2.
\end{equation}
Note that $\vertiii{\cdot}_h$ is indeed a norm on $W_{h,0}$.
This norm will be used in the well-posedness and convergence analysis below.

We first give an immediate result about the consistency of the discrete form \cref{eq:weak-u}.

\begin{theorem}
    \label{thm:consistency}
    (Consistency)
    Assuming that $u\in H^4(\Omega)$.
    The solution $u$ of the continuous weak form \cref{eq:weak-uorig} solves the discrete weak problem \cref{eq:weak-u}.
\end{theorem}
\begin{proof}
    Multiplying the fourth-order term $2B\nabla\cdot(\nabla\cdot (\mathcal{D}^2 u))$ in \cref{eq:strong-u} with a test function $t\in W_{h,0}$ and using piecewise integration by parts with the boundary condition specified in \cref{eq:strong-u} for $u$, one can obtain
    \begin{equation}
        \label{eq:ibp}
        2B \sum_{T\in \mathcal{E}_h} \int_T \nabla\cdot(\nabla\cdot(\mathcal{D}^2 u)) t = 2B \sum_{T\in \mathcal{E}_h} \int_T \mathcal{D}^2 u\colon \mathcal{D}^2 t
        - 2B \sum_{e\in\mathcal{E}_I} \int_e \llbrace \frac{\partial^2 u}{\partial\nu^2}\rrbrace \llbracket \nabla t\rrbracket.
    \end{equation}
    Since $u\in H^4(\Omega)$ implies $\nabla u$ is continuous on the whole domain $\Omega$, the jump term $\llbracket \nabla u \rrbracket$ then becomes zero and we can thus symmetrise and penalise the form \cref{eq:ibp}.
    This leads to the presence of $A^s_h(u,t)+P^s_h(u,t)$.
    The remaining terms involving $B^s$ and $C^s$ are straightforward as one takes the test function $t\in W_{h,0}$.
    Therefore, $u$ satisfies \cref{eq:weak-u}.
\end{proof}


By noting that
$
    \left(\mathcal{D}^2\colon \mathcal{D}^2\right)u = \left[ \left(\partial_x^2\right)^2 +
    \left(\partial_y^2\right)^2 + 2\left(\partial_{xy}^2 \right)^2\right] u
    = \Delta^2 u,
$
it is natural to extend the classical elliptic regularity result \cite{blum-1980-article} for the biharmonic operator $\Delta^2$ to the case of the bi-Hessian operator $\mathcal{D}^2\colon \mathcal{D}^2$.
If the domain $\Omega$ is smooth, the weak solutions of the biharmonic problem (e.g., \cite[Example 2]{brenner-2011-book}) belong to $H^4(\Omega)$ by classical elliptic regularity results and thus we make this assumption henceforth.

Moreover, to facilitate the analysis, we further assume that $u$ is an isolated solution, i.e., there is only one solution $u$ satisfying \cref{eq:strong-u} within a sufficiently small ball $\{v\in H^2(\Omega)\cap H^1_0(\Omega): |v-u|_2\le r_b\}$ with radius $r_b$.
These assumptions then imply that the linearised operator $\langle \mathcal{D}\mathcal{N}^s(u)\cdot,\cdot\rangle_{H^2}$ satisfies the following inf-sup condition:
\begin{equation}
    \label{eq:inf-sup-DN-smec}
    0<\beta_u = \adjustlimits\inf_{\underset{|v|_2=1}{v\in H^2(\Omega)\cap H^1_0(\Omega)}} \sup_{\underset{|w|_2=1}{w\in H^2(\Omega)\cap H^1_0(\Omega)}} \langle \mathcal{D}\mathcal{N}^s(u)v,w\rangle_{H^2}
    =  \adjustlimits\inf_{\underset{|w|_2=1}{w\in H^2(\Omega)\cap H^1_0(\Omega)}} \sup_{\underset{|v|_2=1}{v\in H^2(\Omega)\cap H^1_0(\Omega)}} \langle \mathcal{D}\mathcal{N}^s(u)v,w\rangle_{H^2}.
\end{equation}

\subsubsection{Well-posedness of the discrete form}

Recalling \cite[Eq.\ (3.20)]{brenner-2011-book}, we can obtain for $v,w\in W_{h,0}$,
\begin{equation}
    \label{eq:est-average}
    \sum_{e\in \mathcal{E}_I} \left| \int_e \llbrace \frac{\partial^2 w}{\partial\nu^2}\rrbrace \llbracket \nabla v \rrbracket \right| \lesssim
    \left( \sum_{T\in\mathcal{T}_h} \int_T \mathcal{D}^2 w\colon \mathcal{D}^2 w \right)^{1/2}
    \left( \sum_{e\in \mathcal{E}_I} \frac{1}{h_e^3} \int_e (\llbracket \nabla v\rrbracket )^2 \right)^{1/2},
\end{equation}
as the edge/facet size $h_e<1$.
With the estimate \cref{eq:est-average} at hand, we then apply the Cauchy--Schwarz inequality and use the definition \cref{eq:mesh-norm} of $\vertiii{\cdot}_h$ to obtain the boundedness of $A^s_h(\cdot,\cdot)$ and $P_h^s(\cdot,\cdot)$.
We omit the details of the proofs here and only illustrate the boundedness result for $B^s(\cdot,\cdot,\cdot,\cdot)$ and $C^s(\cdot,\cdot)$ below.
\begin{lemma}
    \label{lem:bound-BandC}
    (Boundedness of $B^s(\cdot,\cdot,\cdot,\cdot)$ and $C^s(\cdot,\cdot)$.)
    For $u,v,w,p \in W_{h,0}$, we have
\begin{equation}
    \label{eq:bdd-Bs-h}
    |B^s(u,v,w,p)| \lesssim \vertiii{u}_h \vertiii{v}_{h} \vertiii{w}_{h}\vertiii{p}_{h} \text{ and }
        |C^s(u,v)| \lesssim \vertiii{u}_{h}\vertiii{v}_{h}.
\end{equation}
For $u,v\in H^2(\Omega)$, $w,p\in W_h$,
\begin{equation}
    \label{eq:bdd-Bs-H2}
    |B^s(u,v,w,p)| \lesssim \| u\|_2 \| v\|_{2} \vertiii{w}_{h}\vertiii{p}_{h}.
\end{equation}
\end{lemma}
\begin{proof}
    By H\"older's inequality, the Sobolev embedding $H^1(\Omega)\hookrightarrow L^4(\Omega)$, and the fact that the $H^1$ semi-norm $|\cdot|_1$ is a norm in $H^1_0(\Omega)$, we deduce
    \begin{equation*}
        \begin{aligned}
            |B^s(u,v,w,p)| &\lesssim \|u\|_{L^4}\|v\|_{L^4}\|w\|_{L^4}\|p\|_{L^4}\\
                             &\lesssim |u|_1 |v|_1 |w|_1 |p|_1.
        \end{aligned}
    \end{equation*} 
    It then follows from a Poincar\'e inequality \cite[Eq.\ (5.7)]{brenner-2004-article} for piecewise $H^2$ functions that
    \begin{equation}
        \label{eq:boundH1byH2}
        \sum_{T\in\mathcal{T}_h} |v|^2_{1,T} \lesssim \sum_{T\in\mathcal{T}_h} |v|^2_{2,T} + \sum_{e\in\mathcal{E}_I} \frac{1}{h_e^3} \|\llbracket \nabla v\rrbracket\|^2_{0,e}=\vertiii{v}^2_h \quad\forall v\in W_{h,0}.
    \end{equation}
    Thus, we obtain
    $
        |B^s(u,v,w,p)| \lesssim \vertiii{u}_h \vertiii{v}_{h} \vertiii{w}_{h}\vertiii{p}_{h}.
    $

    The boundedness of $C^s(\cdot,\cdot)$ follows similarly by the Cauchy--Schwarz inequality, the Sobolev embedding $H^1(\Omega)\hookrightarrow L^2(\Omega)$ and the use of \cref{eq:boundH1byH2}.
    The proof of \cref{eq:bdd-Bs-H2} is analogous to that of \cref{eq:bdd-Bs-h} with a use of the embedding result $H^2(\Omega)\hookrightarrow L^\infty(\Omega)$ and the Cauchy--Schwarz inequality.
\end{proof}

We give the coercivity result for the bilinear form $\left(A^s_h(\cdot,\cdot)+P^s_h(\cdot,\cdot)\right)$.

\begin{lemma}
    (Coercivity of $A^s_h+P^s_h$)
    For a sufficiently large penalty parameter $\epsilon$, there holds
    \begin{equation}
        \label{eq:coer-As-Ps}
        \vertiii{v_h}_h^2 \lesssim A^s_h(v_h,v_h)+P^s_h(v_h,v_h)\quad \forall v_h\in W_{h,0}.
    \end{equation}
\end{lemma}
\begin{proof}
By \cref{eq:est-average} and the inequality of geometric and arithmetic means, we deduce for $v\in W_{h}$,
    \begin{align*}
            A^s_h(v,v)+P^s_h(v,v) &\geq
                2B\sum_{T\in \mathcal{T}_h} |v|_{H^2(T)}^2
                - 2BC \left( \sum_{T\in\mathcal{T}_h} |v|^2_{2,T}\right)^{1/2} \left( \sum_{e\in\mathcal{E}_I} \frac{1}{h_e^3} \| \llbracket \nabla v\rrbracket \|^2_{0,e} \right)^{1/2}
                                                             + 2B \left(\sum_{e\in \mathcal{E}_I} \int_e \frac{\epsilon}{h_e^3} |\llbracket \nabla v \rrbracket |^2\right)\\
                 &\geq
                 2B \left[ \frac{1}{2} \sum_{T\in \mathcal{T}_h} |v|_{H^2(T)}^2
                 + \left(\epsilon - \frac{C^2}{2}\right) \sum_{e\in\mathcal{E}_I} \frac{1}{h_e^3} \|\llbracket \nabla v\rrbracket \|^2_{0,e} \right]
                 \geq B \vertiii{v}_h^2,
    \end{align*}
    provided the penalty parameter $\epsilon$ is sufficiently large with the generic constant $C$ from \cref{eq:est-average}.
\end{proof}

An important question about the well-posedness is the coercivity of the bilinear operator $\langle\mathcal{D}\mathcal{N}^s_h(u_h)\cdot,\cdot\rangle$.
Due to the presence of $B^s$ and $C^s$ terms in $\langle\mathcal{D}\mathcal{N}^s_h(u_h)\cdot,\cdot\rangle$, it is not trivial to derive its coercivity.
We first discuss the weak coercivity of the bilinear form $\langle\mathcal{D}\mathcal{N}^s_h(u)\cdot,\cdot\rangle$ defined as
\begin{equation}
    \langle\mathcal{D}\mathcal{N}^s_h(u)v_h,w_h\rangle \coloneqq A^s_h(v_h,w_h) + P^s_h(v_h,w_h) + 3B^s(u,u,v_h,w_h) + C^s(v_h,w_h) \quad \forall v_h,w_h\in W_{h,0}.
\end{equation}
To this end, we will employ the enrichment operator $E_h: W_h\rightarrow W_C\subset H^2(\Omega)$ with $W_C$ being the Hsieh--Clough--Tocher macro finite element space \cite{brenner-2011-book}.
The following lemma is adapted to our notation and definition of $\vertiii{\cdot}_h$ from \cite[Lemma 1]{brenner-2011-book}.

\begin{lemma}
    \label{lem:enrichment}
    \cite[Lemma 1]{brenner-2011-book}
    For $v_h\in W_{h,0}$, there holds
    \begin{equation*}
        \sum_{T\in\mathcal{T}_h} \left( h^{-4} \|v_h-E_h v_h\|^2_{L^2(T)} + h^{-2}|v_h-E_hv_h|^2_{H^1(T)} + |v_h-E_hv_h|^2_{H^2(T)} \right)
                                \lesssim \sum_{e\in \mathcal{E}_I} \frac{1}{h_e^3}\|\llbracket \nabla v_h\rrbracket \|^2_{L^2(e)} \lesssim \vertiii{v_h}_h^2.
\end{equation*}
\end{lemma}

With this, we obtain the following discrete inf-sup condition for the discrete bilinear operator $\langle\mathcal{D}\mathcal{N}^s_h(u)\cdot,\cdot\rangle$.

\begin{theorem}
    \label{thm:inf-sup-u}
    (Weak coercivity of $\langle\mathcal{D}\mathcal{N}^s_h(u)\cdot,\cdot\rangle$)
    Let $u$ be a regular isolated solution of the nonlinear continuous weak form \cref{eq:weak-u}.
    For a sufficiently large $\epsilon$ and a sufficiently small mesh size $h$, the following discrete inf-sup condition holds with a positive constant $\beta_c>0$:
    \begin{equation}
        \label{eq:discrete-dn}
        0<\beta_c \leq \adjustlimits\inf_{\underset{\vertiii{v_h}_h=1}{v\in W_{h,0}}}\sup_{\underset{\vertiii{w_h}_h=1}{w\in W_{h,0}}} \langle \mathcal{D}\mathcal{N}^s_h(u)v_h,w_h\rangle.
    \end{equation}
\end{theorem}
\begin{proof}
    For $v\in H^2(\Omega)\cap H^1_0(\Omega)$, it follows from the boundedness result of $B^s,C^s$ that $B^s(u,u,v,\cdot)$ and $C^s(v,\cdot)\in (L^2(\Omega))^\star$.
    Furthermore, since $A^s(\cdot,\cdot)$ is bounded and coercive as given by \cref{lem:As-coer-bound}, for a given $v_h\in W_h$ with $\vertiii{v_h}_h=1$, there exists $\xi$ and $\eta\in H^4(\Omega)\cap H^1_0(\Omega)$ that solve the linear systems:
    \begin{subequations}
        \begin{equation}
            \label{eq:As-sub1}
            A^s(\xi,w) = 3B^s(u,u,v_h,w) + C^s(v_h,w) \quad \forall w\in H^2(\Omega)\cap H^1_0(\Omega),
        \end{equation}
        \begin{equation}
            \label{eq:As-sub2}
            A^s(\eta,w) = 3B^s(u,u,E_h v_h,w) + C^s(E_h v_h,w) \quad \forall w\in H^2(\Omega)\cap H^1_0(\Omega).
        \end{equation}
    \end{subequations}
    We require the $H^4$-regularity for the derivation of \cref{eq:asps-sub2} below.
    It then follows from the standard elliptic regularity result that $\|\eta\|_4 \lesssim C_{BC}$ with a constant $C_{BC}$ depending on $\|u\|_2$.

    Subtracting \cref{eq:As-sub1} from \cref{eq:As-sub2}, then taking $w=\eta-\xi$ and using the coercivity of $A^s(\cdot,\cdot)$, boundedness of $B^s(\cdot,\cdot,\cdot,\cdot)$, $C^s(\cdot,\cdot)$ and enrichment estimate \cref{lem:enrichment}, we get
        \begin{equation}
            |\eta-\xi|_2 \lesssim \left(3 \|u\|_2^2 +1\right) \|E_hv_h-v_h\|_0
            \lesssim h^2 \vertiii{v_h}_h = h^2.
                         \label{eq:eta-xi}
        \end{equation}
    Since $u$ is a regular isolated solution of \cref{eq:weak-uorig}, it yields by \cref{eq:inf-sup-DN-smec}, \cref{eq:As-sub2}, \cref{lem:As-coer-bound}, the fact that $\llbracket \nabla(E_hv_h+\eta)\rrbracket=0$ and the triangle inequality, that there exists $w\in H^2(\Omega)\cap H^1_0(\Omega)$ with $|w|_2=1$ such that
        \begin{align}
            |E_h v_h|_2 &\lesssim \langle \mathcal{D}\mathcal{N}^s(u) E_hv_h, w\rangle_{H^2}
                        = A^s(E_h v_h, w)+ 3B^s(u,u,E_hv_h,w)+C^s(E_hv_h,w) \nonumber \\
                        &= A^s(E_hv_h+\eta,w)
                        \lesssim |E_hv_h +\eta|_2 {|w|_2} 
                        = \vertiii{E_hv_h+\eta}_h \nonumber \\
                        &\le \vertiii{E_hv_h-v_h}_h +\vertiii{v_h+I_h\xi}_h+\vertiii{I_h\xi-\xi}_h+ \underbrace{\vertiii{\xi-\eta}_h}_{=|\xi-\eta|_2}.\label{eq:ehvs-2}
        \end{align}
        Note that $\llbracket \nabla\xi\rrbracket=0$ on $\mathcal{E}_I$ since $\xi\in H^4(\Omega)$.
    We can thus calculate using \cref{lem:enrichment} that
    \begin{equation*}
            \vertiii{E_hv_h-v_h}_h^2 \lesssim \sum_{e\in\mathcal{E}_I}\int_e \frac{1}{h_e^3} |\llbracket \nabla v_h\rrbracket |^2  
                                     \lesssim \sum_{e\in\mathcal{E}_I}\int_e \frac{1}{h_e^3} |\llbracket \nabla (v_h +\xi)\rrbracket |^2 
                                     \le \vertiii{v_h+\xi}_h^2.
    \end{equation*}
    Further, by the triangle inequality, we get
    \begin{equation}
        \label{eq:ehvh-sub}
        \vertiii{E_hv_h-v_h}_h \lesssim \vertiii{v_h+\xi}_h \le \vertiii{v_h+I_h\xi}_h + \vertiii{\xi-I_h\xi}_h.
    \end{equation}

    Since $v_h+I_h\xi\in W_h$, it follows from the coercivity result \cref{eq:coer-As-Ps} that there exists $w_h\in W_h$ with $\vertiii{w}_h=1$ such that
        \begin{align}
            \vertiii{v_h+I_h\xi}_h &\lesssim A^s_h(v_h+I_h\xi,w_h) + P^s_h(v_h+I_h\xi,w_h)\nonumber \\
                                   &= \langle \mathcal{D}\mathcal{N}^s_h(u)v_h,w_h\rangle - 3B^s(u,u,v_h,w_h) - C^s(v_h,w_h) \nonumber \\
                                   &\quad + A^s_h(I_h\xi-\xi,w_h) + P^s_h(I_h\xi-\xi,w_h)+ A^s_h(\xi,w_h) + P^s_h(\xi,w_h) \nonumber \\
                                   &=\langle \mathcal{D}\mathcal{N}^s_h(u)v_h,w_h\rangle + 3B^s(u,u,v_h,E_hw_h-w_h) + C^s(v_h,E_hw_h-w_h) \nonumber \\
                                   &\quad + A^s_h(I_h\xi-\xi,w_h) + P^s_h(I_h\xi-\xi,w_h)
                                   + A^s_h(\xi,w_h-E_hw_h) + P^s_h(\xi,w_h-E_hw_h),\label{eq:vhplus-sub}
        \end{align}
    where in the last equality we have used the fact that
    \begin{equation*}
        3B^s(u,u,v_h,E_hw_h) + C^s(v_h,E_hw_h) = A^s(\xi,E_hw_h) = A^s_h(\xi,E_hw_h)+P^s(\xi,E_hw_h)
    \end{equation*}
    because of \cref{eq:As-sub1} and $\llbracket\nabla\xi\rrbracket=\llbracket\nabla E_hw_h\rrbracket=0$.

    Using the boundedness result \cref{lem:bound-BandC} and the enrichment estimate \cref{lem:enrichment}, we obtain
    \begin{equation}
        \label{eq:bscs-sub}
        3B^s(u,u,v_h,E_hw_h-w_h) + C^s(v_h,E_hw_h-w_h) \lesssim \underbrace{\|v_h\|_0}_{\lesssim |v_h|_1\lesssim \vertiii{v_h}_h=1} \underbrace{\|E_hw_h-w_h\|_0}_{\lesssim h^2\vertiii{w_h}_h=h^2}.
    \end{equation}
    By the boundedness of the bilinear form $A^s_h+P^s_h$ and standard interpolation estimate, we have
        \begin{equation}
            A^s_h(I_h\xi-\xi,w_h) +P^s_h(I_h\xi-\xi,w_h) \lesssim \vertiii{I_h\xi-\xi}_h \underbrace{\vertiii{w_h}_h}_{=1}
        \lesssim h^{\min\{\deg-1,2\}}\|\xi\|_4.
        \label{eq:asps-sub}
    \end{equation}
    Moreover, by the fact that $\llbracket \nabla\xi\rrbracket=\llbracket \nabla(E_hw_h)\rrbracket=0$, the enrichment estimate \cref{lem:enrichment} and \cref{eq:ibp}, there holds
        \begin{align}
            A^s_h &(\xi,w_h-E_hw_h) +P^s_h(\xi,w_h-E_hw_h)\nonumber \\
                  &= 2B \sum_{T\in\mathcal{T}_h} \int_T \mathcal{D}^2\xi\colon \mathcal{D}^2(w_h-E_hw_h) - 2B \sum_{e\in\mathcal{E}_I}\int_e\llbrace \frac{\partial^2 \xi}{\partial\nu^2}\rrbrace \llbracket \nabla(w_h-E_hw_h)\rrbracket \nonumber \\
                                                         &= 2B \sum_{T\in\mathcal{T}_h} \nabla\cdot \left(\nabla\cdot(\mathcal{D}^2 \xi)\right) (w_h-E_hw_h)
                                                         \lesssim \|\xi\|_4 \|w_h-E_hw_h\|_0
                                                         \lesssim h^2\|\xi\|_4. \label{eq:asps-sub2}
        \end{align}
    Combine \cref{eq:bscs-sub,eq:asps-sub,eq:asps-sub2} in \cref{eq:vhplus-sub} to obtain
    \begin{equation}
        \label{eq:combine-sub1}
        \vertiii{v_h+I_h\xi}_h \lesssim \langle \mathcal{D}\mathcal{N}^s_h(u)v_h,w_h\rangle + h^2 + h^{\min\{\deg-1,2\}}.
    \end{equation}
    Substituting \cref{eq:combine-sub1} into \cref{eq:ehvh-sub} and using standard interpolation estimates yield that
    \begin{equation}
        \label{eq:ehvh-sub2}
        \vertiii{E_hv_h-v_h}_h \lesssim \langle \mathcal{D}\mathcal{N}^s_h(u)v_h,w_h\rangle + h^2 + h^{\min\{\deg-1,2\}}.
    \end{equation}
    A use of \cref{eq:ehvh-sub2,eq:combine-sub1}, standard interpolation estimates and \cref{eq:eta-xi} in \cref{eq:ehvs-2} leads to
    \begin{equation*}
        |E_hv_h|_2 \lesssim \langle \mathcal{D}\mathcal{N}^s_h(u)v_h,w_h\rangle + h^2 + h^{\min\{\deg-1,2\}}.
    \end{equation*}
    Then, by the triangle inequality, we have
    \begin{equation*}
            1=\vertiii{v_h}_h \le \vertiii{v_h-E_hv_h}_h + \underbrace{\vertiii{E_hv_h}_h}_{=|E_hv_h|_2}
                              \le C_{t} \left( \langle \mathcal{D}\mathcal{N}^s_h(u)v_h,w_h\rangle + h^2 + h^{\min\{\deg-1,2\}} \right).
    \end{equation*}
    Therefore, for the mesh size $h$ satisfying
    $
        h^2 + h^{\min\{\deg-1,2\}} < \frac{1}{2C_t},
    $
    the discrete inf-sup condition \cref{eq:discrete-dn} holds for $\beta_c=\frac{1}{2C_t}$.
\end{proof}

We can now obtain the discrete inf-sup condition for the perturbed bilinear form $\langle\mathcal{D}\mathcal{N}^s_h(I_hu)\cdot,\cdot\rangle$ given by
\begin{equation}
    \label{eq:def-per-bilinear}
    \langle\mathcal{D}\mathcal{N}^s_h(I_hu)v_h,w_h\rangle = A^s_h(v_h,w_h) + P^s_h(v_h,w_h) + 3B^s(I_hu,I_h u, v_h,w_h) + C^s(v_h,w_h) \quad \forall v_h,w_h\in W_{h,0}.
\end{equation}
Here, we employ the interpolation operator $I_h:H^2(\Omega)\cap H^1_b(\Omega;\mathbb{R}) \to W_{h,b}$.

\begin{theorem}
    \label{thm:weak-coer-per}
    (Weak coercivity of $\langle\mathcal{D}\mathcal{N}^s_h(I_h u)\cdot,\cdot\rangle$)
    Let $u$ be a regular isolated solution of the nonlinear continuous weak form \cref{eq:weak-u} and $I_h u$ the interpolation of $u$.
    For a sufficiently large $\epsilon$ and a sufficiently small mesh size $h$, the following discrete inf-sup condition holds:
    \begin{equation}
        \label{eq:weak-coer-per}
        0<\frac{\beta_c}{2} \leq \adjustlimits\inf_{\underset{\vertiii{v_h}_h=1}{v_h\in W_{h,0}}}\sup_{\underset{\vertiii{w_h}_h=1}{w_h\in W_{h,0}}} \langle \mathcal{D}\mathcal{N}^s_h(I_h u)v_h,w_h\rangle.
    \end{equation}
\end{theorem}
\begin{proof}
    Denote $\tilde{u}= u-I_h u$. By the definition \cref{eq:def-per-bilinear} of the bilinear form $\langle\mathcal{D}\mathcal{N}^s_h(I_h u)\cdot,\cdot\rangle$, we have
    $
    \langle\mathcal{D}\mathcal{N}^s_h(I_h u)v_h,w_h\rangle = A^s_h(v_h,w_h) +P^s_h(v_h,w_h) + 3B^s(u-\tilde{u}, u-\tilde{u}, v_h,w_h) + C^s(v_h,w_h).
    $
    It follows from the definition of $B^s$ and its boundedness result \cref{lem:bound-BandC} that
    \begin{equation*}
        \begin{aligned}
            B^s(u-\tilde{u}, u-\tilde{u}, v_h,w_h) &= B^s(u,u,v_h,w_h) + B^s(\tilde{u},\tilde{u},v_h,w_h)-2B^s(u,\tilde{u}, v_h,w_h) \\
                                                     &\ge B^s(u,u,v_h,w_h) + B^s(\tilde{u},\tilde{u},v_h,w_h) - 2C_1 \vertiii{u}_h \vertiii{\tilde{u}}_h \vertiii{v_h}_h \vertiii{w_h}_h,
        \end{aligned}
    \end{equation*}
    where $C_1$ is the generic constant arising in the boundedness result \cref{lem:bound-BandC} for $B^s(\cdot,\cdot,\cdot,\cdot)$.
    Therefore, we obtain that
    \begin{equation*}
         \langle\mathcal{D}\mathcal{N}^s_h(I_h u)v_h,w_h\rangle \ge  \langle\mathcal{D}\mathcal{N}^s_h(u)v_h,w_h\rangle 
         + 3 B^s(\tilde{u},\tilde{u},v_h,w_h) - 6C_1 \vertiii{u}_h \vertiii{\tilde{u}}_h \vertiii{v_h}_h \vertiii{w_h}_h.
    \end{equation*}
    Now using the inf-sup condition \cref{thm:inf-sup-u} for the bilinear form $\langle\mathcal{D}\mathcal{N}^s_h(u)\cdot,\cdot\rangle$, boundedness result \cref{lem:bound-BandC} and interpolation estimates, we get
        \begin{align*}
            \sup_{\underset{w_h\in W_{h,0}}{\vertiii{w_h}_h=1}} \langle\mathcal{D}\mathcal{N}^s_h(I_h u)v_h,w_h\rangle &\ge \sup_{\underset{w_h\in W_{h,0}}{\vertiii{w_h}_h=1}} \langle\mathcal{D}\mathcal{N}^s_h(u)v_h,w_h\rangle \\
                                                                                                                       & \quad - 3|B^s(\tilde{u},\tilde{u},v_h,w_h)|
                                                                                            - 6C_1 h^{\min\{\deg-1, \Bbbk_u-2\}} \vertiii{u}_h \vertiii{v_h}_h \\
                                                                                            & \ge \left(\beta_c - C_2 h^{\min\{\deg-1, \Bbbk_u-2\}} \right) \vertiii{v_h}_h 
                                                                                            \ge \frac{\beta_c}{2} \vertiii{v_h}_h,
    \end{align*}
    for a sufficiently small mesh size $h$ such that $h^{\min\{\deg-1, \Bbbk_u-2\}} < \frac{\beta_c}{2C_2}$.
    Here, $C_2$ depends on $C_1$ and $\|u\|_{\Bbbk_u}$ and $\Bbbk_u\geq 4$ gives the regularity of $u$, i.e., $u\in H^{\Bbbk_u}(\Omega)$.
    Therefore, the inf-sup condition \cref{eq:weak-coer-per} holds.
\end{proof}

\subsubsection{Convergence analysis}
We proceed to the error analysis for the discrete nonlinear problem \cref{eq:weak-u}.
Let
$
    \mathcal{B}_{\rho}(I_h u) \coloneqq \{v_h\in W_{h}: \vertiii{I_h u - v_h}_{h} \le \rho \}.
$
We define the nonlinear map $\mu_h: W_{h}\to W_{h}$ by
\begin{equation}
    \label{eq:nonlinearmap}
    \langle \mathcal{D}\mathcal{N}^s_h(I_h u) \mu_h(v_h), w_h\rangle = 3B^s(I_h u, I_h u, v_h, w_h) + L^s(w_h) - B^s(v_h,v_h,v_h,w_h)
\end{equation}
for $v_h,w_h\in W_{h,0}$.
Due to the weak coercivity property in \cref{thm:weak-coer-per}, the nonlinear map $\mu_h$ is well-defined.

The existence and local uniqueness of the solution $u_h$ to the discrete nonlinear problem \cref{eq:weak-u} will be proven via an application of Brouwer's fixed point theorem, which necessitates the use of two auxiliary lemmas illustrating that (i) $\mu_h$ maps from a ball to itself; and (ii) the map $\mu_h$ is contracting.

%

\begin{lemma}
    \label{lem:mappingball}
    (Mapping from a ball to itself)
    Let $u$ be a regular isolated solution of the continuous nonlinear weak problem \cref{eq:weak-uorig}.
    For a sufficiently large $\epsilon$ and a sufficiently small mesh size $h$, there exists a positive constant $R(h)>0$ such that:
    \begin{equation*}
        \vertiii{v_h - I_h u}_{h} \le R(h) \Rightarrow \vertiii{\mu_h(v_h)- I_h u}_{h} \le R(h)\quad \forall v_h\in W_{h,0}.
    \end{equation*}
\end{lemma}
\begin{proof}
    Note that the solution $u\in H^2(\Omega)\cap H^1_0(\Omega)$ of \cref{eq:weak-uorig} satisfies the discrete weak formulation \cref{eq:weak-u} due to the consistency result \cref{thm:consistency}, that is to say, there holds that
    \begin{equation}
        \label{eq:discrete-u-problem}
        A_h^s(u, w_h) + P_h^s(u,w_h) + B^s(u,u,u,w_h) + C^s(u,w_h) = L^s(w_h)
        \quad\forall w_h\in W_{h,0}.
    \end{equation}
    By the linearity of $\langle \mathcal{D}\mathcal{N}^s_h(I_h u) \cdot,\cdot\rangle_{H^2}$, the definition \cref{eq:nonlinearmap} of the nonlinear map $\mu_h$ and \cref{eq:discrete-u-problem}, we calculate
        \begin{align*}
            \langle \mathcal{D} &\mathcal{N}^s_h(I_h u) (I_h u - \mu_h(v_h)), w_h\rangle
                            = \langle \mathcal{D}\mathcal{N}^s_h(I_h u) I_h u, w_h\rangle - \langle \mathcal{D}\mathcal{N}^s_h(I_h u) \mu_h(v_h), w_h\rangle \\
                            &= {A_h^s(I_h u, w_h) + P_h^s(I_h u,w_h)} + 3B^s(I_h u,I_h u,I_h u,w_h) + C^s(I_h u,w_h)\\
                  &\quad - 3B^s(I_h u, I_h u, v_h, w_h) + B^s(v_h,v_h,v_h,w_h) - L^s(w_h)\\
                  &= \underbrace{A^s_h(I_h u- u, w_h)+ P^s_h(I_h u - u, w_h)}_{\eqqcolon \mathfrak{N}_1} + \underbrace{C^s(I_h u -u, w_h)}_{\eqqcolon \mathfrak{N}_2}
                  + \underbrace{\left(B^s(I_h u, I_h u, I_h u, w_h) - B^s(u, u,u,w_h)\right)}_{\eqqcolon \mathfrak{N}_3} \\
                  &\quad + \underbrace{\left( 2B^s(I_h u, I_h u, I_h u, w_h) - 3B^s(I_h u, I_h u, v_h, w_h) + B^s(v_h,v_h, v_h, w_h)\right)}_{\eqqcolon \mathfrak{N}_4}
    \end{align*}
    In what follows, we give upper bounds for each $\mathfrak{N}_i, i=1,2,3,4$.
    Using the boundedness of $A^s_h+P^s_h, C^s$ and the interpolation estimate \cite[Eq.\ (5.3)]{brenner-2005-article} in the $\vertiii{\cdot}$-norm, we obtain
    \begin{equation*}
        \begin{aligned}
        &\mathfrak{N}_1 \lesssim \vertiii{I_h u -u}_h \vertiii{w_h}_{h} \lesssim h^{\min\{\mathrm{deg}-1, \Bbbk_u-2\}} \vertiii{w_h}_{h},\\
        &\mathfrak{N}_2 \lesssim\vertiii{I_h u -u}_h \vertiii{w_h}_{h} \lesssim h^{\min\{\mathrm{deg}-1, \Bbbk_u-2\}} \vertiii{w_h}_{h}.\\
        \end{aligned}
    \end{equation*}
    We rearrange terms in $\mathfrak{N}_3$ and use the boundedness result \cref{lem:bound-BandC} and the interpolation result \cite[Eq.\ (5.3)]{brenner-2005-article} to obtain
        \begin{align*}
            \mathfrak{N}_3 
                           &= B^s(I_h u-u, I_h u-u, I_h u, w_h) + 2B^s(I_h u-u, I_h u-u, u, w_h) + 3B^s(u, u, I_h u-u, w_h)\\
                           &\lesssim \left(\vertiii{I_h u-u}_h^2\vertiii{I_hu}_h 
                           + \vertiii{I_hu-u}_h^2\vertiii{u}_h + \|u\|_2^2\|I_hu-u\|_0 \right) \vertiii{w_h}_{h}\\
                           &\lesssim \left( h^{2\min\{\mathrm{deg}-1,\Bbbk_u-2\}} + h^{\min\{\mathrm{deg}+1,\Bbbk_u\}}\right) \vertiii{w_h}_{h}.
    \end{align*}
    Let $e_I = v_h - I_h u$. We use the definition of $B^s(\cdot,\cdot,\cdot,\cdot)$ and its boundedness to deduce that
        \begin{align*}
            \mathfrak{N}_4 
                           &= a_3 \int_\Omega \left\{ 2(I_h u)^3 w_h -3(I_hu)^2v_hw_h +v_h^3 w_h\right\}\\
                           &= a_3\int_\Omega \left\{ \left(v_h^2-(I_hu)^2\right)v_hw_h + 2(I_hu)^2 (I_hu - v_h)w_h \right\}\\
                           &= a_3\int_\Omega \left\{ e_I(e_I+2I_hu)(e_I+I_hu)w_h -2(I_hu)^2 e_Iw_h \right\}\\
                           &= a_3\int_\Omega \left\{e_I \left(e_I^2 +3e_I I_hu + 2 (I_hu)^2\right)w_h - 2 (I_hu)^2 e_I w_h\right\} \\
                           &= a_3\int_\Omega \left(e_I^3 +3e_I^3 I_hu\right)w_h
                           = B^s(e_I,e_I,e_I,w_h)+3B^s(e_I,e_I,I_hu,w_h) \\
                           &\lesssim \vertiii{e_I}_h^2 \left(\vertiii{e_I}_h+\vertiii{I_hu}_h\right) \vertiii{w_h}_h.
        \end{align*}
    Hence, we combine the above bounds for $\mathfrak{N}_i$, $i=1,2,3,4$ to have
        \begin{equation*}
        \langle D \mathcal{N}^s_h(I_h u) (I_h u - \mu_h(v_h)), w_h\rangle
        \lesssim \left(h^{\min\{\mathrm{deg}-1,\Bbbk_u-2\}} + h^{\min\{2\mathrm{deg}-2,2\Bbbk_u-4,\mathrm{deg}+1,\Bbbk_u\}} + \vertiii{e_I}_h^2 \left(\vertiii{e_I}_h+1\right) \right) \vertiii{w_h}_h.
        \end{equation*}
    By the inf-sup condition \cref{eq:weak-coer-per} for the perturbed bilinear form, we further deduce that there exists a $w_h\in W_h$ with $\vertiii{w_h}_h=1$ such that
    $
        \vertiii{I_hu - \mu_h(v_h)}_h \lesssim \langle D \mathcal{N}^s_h(I_h u) (I_h u - \mu_h(v_h)), w_h\rangle.
    $
    Since $\vertiii{e_I}_h\le R(h)$, we obtain
        \begin{align*}
            &\vertiii{I_hu - \mu_h(v_h)}_h \lesssim \left(h^{\min\{\mathrm{deg}-1,\Bbbk_u-2\}} + h^{\min\{2\mathrm{deg}-2,2\Bbbk_u-4,\mathrm{deg}+1,\Bbbk_u\}} + R(h)^2 \left(R(h)+1\right) \right)\\
                                          &\quad \le 
                                          \begin{cases}
                                              C_u\left( 2h^{\min\{\mathrm{deg}-1,\Bbbk_u-2\}} + R(h)^2(1+R(h))\right) & \text{for }2\le \mathrm{deg}\le 3,\Bbbk_u\le 4,\\
                                              C_u \left( h^{\min\{\mathrm{deg}-1,\Bbbk_u-2\}} + h^{\min\{\mathrm{deg}+1,2\Bbbk_u-4\}} + R(h)^2(1+R(h))\right) &\text{for }\mathrm{deg}> 3,\Bbbk_u\le 4.
                                          \end{cases}
        \end{align*}
    Note that there are other cases when $\Bbbk_u >4$ and we only focus on the case of $\Bbbk_u\le 4$ here for brevity.
    The idea of the remainder of the proof is to choose an appropriate $R(h)$ so that $\vertiii{I_hu - \mu_h(v_h)}_h\le R(h)$.
    For simplicity of the calculation, we illustrate the case when $2\le \mathrm{deg}\le 3,\Bbbk_u\le 4$.
    To this end, we take $R(h)= 4C_u h^{\min\{\mathrm{deg}-1,\Bbbk_u-2\}}$ and choose $h$ satisfying
$
    h^{2\min\{\mathrm{deg}-1,\Bbbk_u-2\}} \le \frac{1}{32C_u}-\frac{1}{16}.
$
This yields
        \begin{align*}
            \vertiii{I_hu - \mu_h(v_h)}_h &\le 2C_u h^{\min\{\mathrm{deg}-1,\Bbbk_u-2\}} \left(1+C_u R(h)^2 + C_u \right) \\
                                          &= 2C_u h^{\min\{\mathrm{deg}-1,\Bbbk_u-2\}} \left(1+ 32C_u^3 h^{2\min\{\mathrm{deg}-1,\Bbbk_u-2\}} + 2C_u \right)
                                          \le R(h).
        \end{align*}
This completes the proof.
\end{proof}

\begin{lemma}
    \label{lem:contraction}
    (Contraction result)
    For a sufficiently large $\epsilon$, a sufficiently small mesh size $h$ and any $v_1,v_2\in \mathcal{B}_{R(h)}(I_h u)$, there holds
    \begin{equation}
        \vertiii{\mu_h(v_1) - \mu_h(v_2)}_{h} \lesssim h^{\min\{\mathrm{deg}-1,\Bbbk_u-2\}} \vertiii{v_1-v_2}_h.
    \end{equation}
\end{lemma}
\begin{proof}
    For $w_h \in W_h$, we use the definition \cref{eq:nonlinearmap} of the nonlinear map $\mu_h$, the definition \cref{eq:def-per-bilinear} and linearity of $\langle \mathcal{D}\mathcal{N}_h^s(I_h u)\cdot, \cdot\rangle$ to calculate
        \begin{align*}
            \langle &\mathcal{D}\mathcal{N}^s_h (I_h u)(\mu_h(v_1) - \mu_h(v_2)), w_h\rangle \\
            &= 3B^s(I_h u, I_h u, v_1, w_h) - B^s(v_1,v_1,v_1,w_h)
            - 3 B^s(I_h u, I_h u, v_2, w_h) + B^s(v_2,v_2,v_2,w_h)\\
            &= a_3\int_\Omega \left( 3(I_hu)^2v_1w_h -v_1^3 w_h\right) - a_3\int_\Omega \left(3(I_hu)^2v_2w_h -v_2^3 w_h\right)\\
            &= a_3\int_\Omega \left(\left((I_hu)^2-v_1^2\right)v_1w_h + 2(I_hu)^2 (v_1-v_2)w_h - \left((I_hu)^2-v_2^2\right)v_2w_h\right) \\
            &= a_3\int_\Omega ( (I_hu-v_1)(v_1-I_hu)(v_1-v_2) w_h + 2(I_hu-v_1)I_hu(v_1-v_2)w_h 
            + (I_hu-v_1)(I_hu+v_1)v_2w_h )\\
            &\quad +2a_3\int_\Omega \left(I_h u(v_1-v_2)(I_hu-v_2)w_h + I_hu(v_1-v_2)v_2w_h\right)
            - a_3\int_\Omega (I_hu-v_2)(I_hu+v_2)v_2w_h\\
            &= a_3 \int_\Omega (I_hu-v_1)(v_1-I_hu)(v_1-v_2)w_h + 2a_3\int_\Omega (I_hu-v_1)I_hu (v_1-v_2)w_h\\
            &\quad +2a_3\int_\Omega (I_hu-v_2)I_hu(v_1-v_2)w_h
            +a_3\int_\Omega (v_1-v_2)\left((I_hu-v_1)+(I_hu-v_2)\right) \left((v_2-I_hu)+I_hu\right)w_h.
        \end{align*}
    Let $e_1=I_hu-v_1$, $e_2=I_hu-v_2$ and $e=v_1-v_2$.
    We make some elementary manipulations and use the boundedness of $B^s$ and the inequality of geometric and arithmetic means to yield
        \begin{align*}
            \langle &\mathcal{D}\mathcal{N}^s_h (I_h u)(\mu_h(v_1) - \mu_h(v_2)), w_h\rangle \\
                                               &= a_3\int_\Omega (-e_1^2)e w_h + 2a_3\int_\Omega e_1(I_hu)e w_h + 2a_3\int_\Omega e_2 (I_hu) ew_h
                                               + a_3\int_\Omega \{ e w_h (e_1 I_h u + e_2 I_h u - e_1e_2 -e_2^2)\}\\
                                               &\lesssim \left(\vertiii{e_1}_h^2 +\vertiii{I_hu}_h\vertiii{e_1}_h + \vertiii{e_2}_h\vertiii{I_hu}_h + \vertiii{e_1}_h\vertiii{e_2}_h + \vertiii{e_2}_h^2\right) \vertiii{e}_h\vertiii{w_h}_h\\
                                               &\lesssim \left( \vertiii{e_1}_h^2 + \vertiii{e_2}_h^2 + \vertiii{e_1}_h+\vertiii{e_2}_h\right)\vertiii{e}_h\vertiii{w_h}_h
                                               \lesssim \left(R(h)^2+R(h)\right) \vertiii{e}_h\vertiii{w_h}_h.
        \end{align*}
    By the inf-sup condition \cref{eq:weak-coer-per}, we know that there exists $w_h\in W_h$ with $\vertiii{w_h}_h=1$ such that
    \begin{equation*}
        \frac{\beta_c}{2} \vertiii{\mu_h(v_1) - \mu_h(v_2)}_h \lesssim \langle \mathcal{D}\mathcal{N}^s_h (I_h u)(\mu_h(v_1) - \mu_h(v_2)), w_h\rangle.
    \end{equation*}
    Therefore, we have
    $
        \vertiii{\mu_h(v_1) - \mu_h(v_2)}_h \lesssim R(h)(1+R(h))\vertiii{e}_h.
    $
    Noting that $R(h)(1+R(h))<1$ for $0<R(h)<\frac{1}{2}$ completes the proof.
\end{proof}

The existence and local uniqueness of the discrete solution $u_h$ can now be obtained via the application of Brouwer's fixed point theorem \cite{kesavan-1989-book}.

\begin{theorem}
    \label{thm:main1u}
    (Convergence in $\vertiii{\cdot}_h$-norm)
    Let $u$ be a regular isolated solution of the nonlinear problem \cref{eq:weak-uorig}.
    For a sufficiently large $\epsilon$ and a sufficiently small $h$, there exists a unique solution $u_h$ of the discrete nonlinear problem \cref{eq:weak-u} within the local ball $\mathcal{B}_{R(h)}(I_h u)$.
    Furthermore, we have
    $
        \vertiii{ u - u_h }_h \lesssim h^{\min\{\deg-1,\Bbbk_u-2\}}.
    $
\end{theorem}
\begin{proof}
    A use of \cref{lem:mappingball} yields that the nonlinear map $\mu_h$ maps a closed convex set $\mathcal{B}_{R(h)}(I_h u)\subset W_h$ to itself.
    Moreover it is a contracting map.
    Therefore, an application of Brouwer's fixed point theorem yields that $\mu_h$ has at least one fixed point, say $u_h$, in this ball $\mathcal{B}_{R(h)}(I_h u)$.
    The uniqueness of the solution to \cref{eq:weak-u} in that ball $\mathcal{B}_{R(h)}(I_h u)$ follows from the contraction result in \cref{lem:contraction}.
    Meanwhile, we have by \cref{lem:mappingball} that
    \begin{equation}
        \label{eq:difference}
        \vertiii{u_h - I_h u}_h \lesssim h^{\min\{\deg-1,\Bbbk_u-2\}}.
    \end{equation}
    The error estimate is then obtained straightforwardly using the triangle inequality
    $
        \vertiii{u-u_h}_h \le \vertiii{u-I_hu}_h + \vertiii{I_hu-u_h}_h
    $
    combined with \cref{eq:difference} and the interpolation estimate \cite[Eq.\ (5.3)]{brenner-2005-article}.
\end{proof}

It follows from \cref{thm:main1u} that optimal convergence rates are achieved in the mesh-dependent norm $\vertiii{\cdot}_h$.
This will be numerically verified in \cref{sec:num}.

\subsubsection{Estimates in the $L^2$-norm}

We derive an $L^2$ error estimate using a duality argument in this subsection.
To this end, we consider the following linear dual problem to the primal nonlinear problem \cref{eq:strong-u}:
\begin{equation}
    \label{eq:strong-dual-u}
    \begin{cases}
        2B \nabla\cdot(\nabla\cdot (\mathcal{D}^2\chi)) + a_1 \chi + 3a_3 u^2\chi = f_{dual} \quad & \text{in }\Omega,\\
        \chi = 0,\quad 
        \mathcal{D}^2\chi=\mathbf{0} \quad &\text{on }\partial\Omega,
    \end{cases}
\end{equation}
for $f_{dual}\in L^2(\Omega)$.
For smooth domains $\Omega$, it can be deduced by a classical elliptic regularity result that $\chi\in H^4(\Omega)$.
The corresponding weak form is: find $\chi\in H^2(\Omega)\cap H^1_0(\Omega)$ such that
\begin{equation*}
    2B\int_\Omega \mathcal{D}^2 \chi\colon \mathcal{D}^2 v + a_1\int_\Omega \chi v + 3a_3\int_\Omega u^2\chi v = \int_\Omega f_{dual} v\quad \forall v\in H^2(\Omega)\cap H^1_0(\Omega),
\end{equation*}
that is to say,
\begin{equation}
    \label{eq:weak-dual-u}
    \langle \mathcal{D}\mathcal{N}^s(u) \chi,v\rangle_{H^2} = \langle \mathcal{D}\mathcal{N}^s_h(u)\chi,v\rangle = ( f_{dual},v)_0.
\end{equation}
\begin{remark}
    The first equality in \cref{eq:weak-dual-u} holds since $u\in H^2(\Omega), \chi\in H^2(\Omega)$ and $v\in H^2(\Omega)$.
\end{remark}

We give two auxiliary results in the following.

\begin{lemma}
    \label{lem:H4-As}
    For $u\in H^{\Bbbk_u}(\Omega)$, $\Bbbk_u>2$, $\chi\in H^4(\Omega)\cap H^1_0(\Omega)$ and $I_h u \in W_{h,0}\subset H^1_0(\Omega)$, there holds that
    \begin{equation*}
        A^s_h(I_h u-u, \chi) + P^s_h(I_h u -u,\chi) \lesssim h^{\min\{\deg+1,\Bbbk_u\}} \|\chi\|_4.
    \end{equation*}
\end{lemma}
\begin{proof}
    Note that $\llbracket \nabla\chi\rrbracket=0$ since $\chi\in H^4(\Omega)$ and $\chi=0$ on $\partial\Omega$.
    We calculate
        \begin{align*}
            A^s(I_h u-u, \chi)+P^s_h(I_h u-u,\chi) 
               &= \sum_{T\in\mathcal{E}_h}\int_T 2B \mathcal{D}^2 (I_h u -u)\colon \mathcal{D}^2 \chi 
                                                   - 2B\sum_{e\in \mathcal{E}_I} \llbrace \frac{\partial^2 (I_hu-u)}{\partial\nu^2}\rrbrace \llbracket \nabla \chi\rrbracket \\
               &\quad - 2B \sum_{e\in \mathcal{E}_I} \llbrace \frac{\partial^2 \chi}{\partial\nu^2}\rrbrace \llbracket \nabla (I_hu-u)\rrbracket
                                                   +\sum_{e\in \mathcal{E}_I}\frac{2B \epsilon}{h_e^3}\int_{e} \llbracket \nabla (I_hu-u) \rrbracket \llbracket \nabla \chi \rrbracket\\
                                                   &= \sum_{T\in\mathcal{E}_h}\int_T 2B \mathcal{D}^2 (I_h u -u)\colon \mathcal{D}^2 \chi - 2B \sum_{e\in \mathcal{E}_I} \llbrace \frac{\partial^2 \chi}{\partial\nu^2}\rrbrace \llbracket \nabla (I_hu-u)\rrbracket\\
                                                   &= \sum_{T\in\mathcal{E}_h}\int_T 2B (I_hu-u)\nabla\cdot(\nabla\cdot(\mathcal{D}^2 \chi)) \\
                                                   &\lesssim \|I_h u -u\|_0 \|\nabla\cdot(\nabla\cdot(\mathcal{D}^2 \chi))\|_0
                                                   \lesssim h^{\min\{\deg+1,\Bbbk_u\}}\|\chi\|_4.
        \end{align*}
    Here, the last, second last, and third last steps follow from standard interpolation estimates, the Cauchy--Schwarz inequality, and integration by parts twice, respectively.
\end{proof}

\begin{lemma}
    \label{lem:frhs-bound}
    The solution $\chi$ of the linear dual problem \cref{eq:strong-dual-u} belongs to $H^4(\Omega)$ on a smooth domain $\Omega$ and it holds that
    \begin{equation}
        \label{eq:bdd-chi}
        \|\chi\|_4 \lesssim \|f_{dual}\|_0.
    \end{equation}
\end{lemma}
\begin{proof}
    We can use the inf-sup condition \cref{eq:inf-sup-DN-smec} for the linear operator $\langle \mathcal{D}\mathcal{N}^s(u)\cdot, \cdot\rangle$, the weak form \cref{eq:weak-dual-u} and the Cauchy--Schwarz inequality to obtain
    \begin{equation}
        \label{eq:chi2}
        |\chi |_2 \lesssim \sup_{\overset{w\in H^2\cap H^1_0}{|w|_2=1}} \langle \mathcal{D}\mathcal{N}^s(u) \chi, w \rangle_{H^2} =\sup_{\overset{w\in H^2\cap H^1_0}{|w|_2=1}} (f_{dual},w)_0 \lesssim \|f_{dual}\|_0 \underbrace{\|w\|_0}_{\lesssim |w|_2=1}.
    \end{equation}
    By the form of \cref{eq:weak-dual-u}, the boundedness of $B^s(u,u,\cdot,\cdot)$ and $C^s(\cdot,\cdot)$, and \cref{eq:chi2}, we have
    \begin{equation}
        \label{eq:hesheschi}
        \begin{aligned}
            \|\nabla\cdot(\nabla\cdot (\mathcal{D}^2\chi))\|_0 &= \|-3B^s(u,u,\chi,\cdot) - C^s(\chi,\cdot)+(f_{dual},\cdot)_0\|_0
                                                               \lesssim {\|\chi\|_0}+\|f_{dual}\|_0
                                                               \lesssim \|f_{dual}\|_0.
    \end{aligned}
    \end{equation}
    Using a bootstrapping argument in elliptic regularity (see, e.g., \cite[Section 6.3]{evans-2010-book}), we can deduce that $\chi\in H^4(\Omega)$ in a smooth domain $\Omega$.
    The regularity estimate \cref{eq:bdd-chi} follows from \cref{eq:hesheschi}.
    \end{proof}

We are ready to derive the $L^2$ a priori error estimates.
\begin{theorem}
    \label{thm:l2-u}
    ($L^2$ error estimate)
    Under the same conditions as \cref{thm:main1u} and assuming further that $\deg > 1$ (since the problem is fourth-order), the discrete solution $u_h$ approximates $u$ such that
    \begin{equation}
        \label{eq:l2-u}
        \|u-u_h\|_0 \lesssim 
        \begin{cases}
            h^{\min\{\deg+1,\Bbbk_u\}} &\text{for } \deg \ge 3,\\
            h^{2\min\{\deg-1,\Bbbk_u-2\}} & \text{for }\deg =2. 
        \end{cases} 
    \end{equation} 
\end{theorem}
\begin{proof}
    Taking $f_{dual}= I_hu-u_h\in W_h\subset H^1(\Omega)\cap H^2(\mathcal{T}_h)$ in \cref{eq:strong-dual-u} and multiplying \cref{eq:strong-dual-u} by a test function $v_h=I_hu-u_h$ and integrating by parts, we obtain
    $
        \langle \mathcal{D}\mathcal{N}^s_h(u) \chi,I_hu-u_h\rangle = \|I_hu-u_h\|_0^2.
    $
    It follows from the fact that $u\in H^{\Bbbk_u}(\Omega)$, $\Bbbk_u\ge 4$, and the definition \cref{eq:weak-uorig} of the nonlinear continuous weak form $\mathcal{N}^s(u)\cdot$ that
        \begin{align*}
            \| I_hu-u_h\|_0^2 &= \langle \mathcal{D}\mathcal{N}^s_h(u) \chi,I_hu-u_h\rangle
            + \mathcal{N}^s_h(u_h)(I_h\chi) -\mathcal{N}^s_h(u) (I_h\chi)\\
                              &= A^s_h(\chi,I_hu-u_h) +P^s_h(\chi,I_hu-u_h) + C^s(\chi,I_hu-u_h) + 3B^s(u,u,\chi,I_hu-u_h) \\
                              &\quad + A^s_h(u_h,I_h\chi)+P^s_h(u_h,I_h\chi) + C^s(u_h,I_h\chi)+B^s(u_h,u_h,u_h,I_h\chi)\\
                              &\quad - A^s_h(u,I_h\chi) - P^s_h(u,I_h\chi) -C^s(u,I_h\chi) -B^s(u,u,u,I_h\chi) \\
                              &= \underbrace{A^s_h(I_hu-u,\chi)+A^s_h(u-u_h,\chi-I_h\chi) + P^s_h(I_hu-u,\chi) +P^s_h(u-u_h,\chi-I_h\chi)}_{\eqqcolon \mathfrak{U}_1}\\
                              &\quad + \underbrace{C^s(I_hu-u,\chi)+C^s(u-u_h,\chi-I_h\chi)}_{\eqqcolon \mathfrak{U}_2}\\
                              &\quad + \underbrace{3B^s(u,u,I_hu-u_h,\chi-I_h\chi) + 3B^s(u,u,I_hu-u,I_h\chi)}_{\eqqcolon \mathfrak{U}_3}\\
                              &\quad+ \underbrace{B^s(u_h,u_h,u_h,I_h\chi)-3B^s(u,u,u_h,I_h\chi) + 2B^s(u,u,u,I_h\chi)}_{\eqqcolon \mathfrak{U}_4}.
        \end{align*}
        We bound each $\mathfrak{U}_i$ separately using the boundedness of $A^s_h,P^s_h,B^s$ and $C^s$, \cref{thm:main1u} and standard interpolation estimates.
    This leads to
        \begin{align*}
            &\mathfrak{U}_1 \lesssim h^{\min\{\deg+1,\Bbbk_u\}}\|\chi\|_4 +
            {\vertiii{u-u_h}_h}
            {\vertiii{\chi-I_h\chi}_h}
                           \lesssim h^{\min\{\deg+1,\Bbbk_u\}}\|\chi\|_4,\\
            &\mathfrak{U}_2 \lesssim
            {\|I_hu-u\|_0}
            {\|\chi\|_0}
            + \vertiii{u-u_h}_h\vertiii{\chi-I_h\chi}_h
                           \lesssim h^{\min\{\deg+1,\Bbbk_u\}}\|\chi\|_4,\\
            &\mathfrak{U}_3
                           \lesssim \|u\|_2^2
                           {\vertiii{I_hu-u_h}_h}
                           {\vertiii{\chi-I_h\chi}_h}
                           + \|u\|_2^2
                           {\|I_hu-u\|_0}
                           {\|I_h \chi\|_0}
                           \lesssim h^{\min\{\deg+1,\Bbbk_u\}}\|\chi\|_4.
        \end{align*}
        Setting $e_3=u_h-u$ and estimating $\mathfrak{U}_4$ as in $\mathfrak{R}_4$ of \cref{lem:mappingball} with the use of \cref{thm:main1u} and $\vertiii{I_h\chi}_h\lesssim \|\chi\|_2\le\|\chi\|_4$ yields
        \begin{equation*}
            \mathfrak{U}_4 \lesssim \vertiii{e_3}_h^2 \left(\vertiii{e_3}_h+\vertiii{u}_h\right) {\vertiii{I_h\chi}_h}
                           \lesssim h^{2\min\{\deg-1,\Bbbk_u-2\}} (h^{\min\{\deg-1,\Bbbk_u-2\}}+1) \|\chi\|_4.
    \end{equation*}
    Combining the above estimates for $\mathfrak{U}_i$ ($i=1,2,3,4$) and using the regularity estimate \cref{eq:bdd-chi} and $\|\chi\|_4\lesssim \|I_hu-u_h\|_0$, we obtain
    \begin{equation*}
        \|I_hu-u_h\|_0 \lesssim 
        \begin{cases}
            h^{\min\{\deg+1,\Bbbk_u\}}  & \text{if }\deg\ge 3,\\
            h^{2\min\{\deg-1,\Bbbk_u-2\}} &\text{if }\deg=2.
        \end{cases}
    \end{equation*}
    Hence, \cref{eq:l2-u} follows from the triangle inequality and standard interpolation estimates.
\end{proof}

\cref{thm:l2-u} implies that for quadratic approximations to the sufficiently regular solution of \cref{eq:strong-u}, there is a sub-optimal convergence rate in the $L^2$-norm while for higher order ($\ge 3$) approximations, we expect optimal $L^2$ error rates.
We shall see numerical verifications of this in the subsequent sections.

\subsubsection{The inconsistent discrete form}

The above analysis considers the consistent weak formulation \cref{eq:weak-u}.
In practice, Xia et al.\ \cite{xia-2021-article} adopted the inconsistent discrete weak form in the implementation due to its cheaper assembly cost: find $u_h\in W_{h,b}$ such that
\begin{equation}
\label{eq:weak-u-actual}
\tilde{\mathcal{N}}_h^s(u_h)t_h = \tilde{A}^s_h(u_h,t_h) + B^s(u_h,u_h,u_h,t_h) + C^s(u_h,t_h) + P^s_h(u_h,t_h)
    = 0 \quad \forall t_h\in W_{h,0},
\end{equation}
where
$
\tilde{A}^s_h(u,t) \coloneqq 2B\sum_{T\in\mathcal{T}_h} \int_T \mathcal{D}^2 u \colon \mathcal{D}^2 t.
$
Comparing $\tilde{A}^s_h$ and $A^s_h$, the missing terms are the interior facet integrals arising from piecewise integration by parts and symmetrisation.
Due to the absence of these terms in $\tilde{A}^s_h$, one can immediately notice that the discrete weak formulation \cref{eq:weak-u-actual} is inconsistent in the sense that the solution $u$ of the strong form \cref{eq:strong-u} does not satisfy the weak form \cref{eq:weak-u-actual}, as opposed to the result of \cref{thm:consistency}.

Despite this inconsistency, in practice this also leads to a convergent numerical scheme with similar convergence rates, as illustrated in \cref{sec:num}.
This is not surprising; a similar idea has also been applied and introduced as \emph{weakly over-penalised symmetric interior penalty} (WOPSIP) methods in \cite{brenner-2008-article} for second-order elliptic PDEs and in \cite{brenner-2010-article} for biharmonic equations.

\begin{remark}
    The excessive size of the penalty parameter in the WOPSIP method could induce ill-conditioned linear systems. No such effects are observed in our numerical results.
\end{remark}

\subsection{A priori error estimates for $(\mathcal{P}2)$}
\label{sec:apriori-Q}

Problem $(\mathcal{P}2)$ is a special form of the classical LdG model of nematic LC.
Finite element analysis for a more general form using conforming discretisations has been studied in \cite{davis-phd-thesis, davis-1998-article}. 
More specifically, Davis and Gartland \cite{davis-1998-article} gave an abstract nonlinear finite element convergence analysis where an optimal $H^1$ error bound is proved on convex domains with piecewise linear polynomial approximations, but do not derive an error bound in the $L^2$ norm.
Recently, Maity, Majumdar and Nataraj \cite{maity-2020a-article} analysed the discontinuous Galerkin finite element method for a two-dimensional reduced LdG free energy, where optimal a priori error estimates in the $L^2$-norm with exact solutions in $H^2$ are achieved for a piecewise linear discretisation.
Both works only focus on piecewise linear approximations.
In this section, we will follow similar steps to \cref{sec:apriori-u} to prove the $H^1$- and $L^2$-convergence rates for the problem $(\mathcal{P}2)$ with the use of common continuous Lagrange elements of arbitrary positive degree.
Since the approach is similar to the previous subsections, we omit some details for brevity.


The continuous weak formulation of $(\mathcal{P}2)$ in two dimensions (the three-dimensional case can be tackled similarly) is given by:
find $\Qvec\in H^1_b(\Omega,S_0)$ such that
\begin{equation}
    \label{eq:weak-Q}
    \mathcal{N}^n(\Qvec)\Pvec \coloneqq A^n(\Qvec,\Pvec) + B^n(\Qvec,\Qvec,\Qvec,\Pvec)+C^n(\Qvec,\Pvec)=0 \quad \forall \Pvec\in \mathbf{H}^1_0(\Omega),
\end{equation}
where the bilinear forms are
$
        A^n(\Qvec,\Pvec) \coloneqq K \int_\Omega \nabla \Qvec\fcolon \nabla \Pvec,\
        C^n(\Qvec,\Pvec) \coloneqq -2l \int_\Omega \Qvec: \Pvec,
$
and the nonlinear operator is given by
\begin{equation}
    \label{eq:def-Bn}
    B^n(\Psi, \Phi,\Theta, \Xi) \coloneqq \frac{4l}{3} \int_\Omega \left( (\Psi : \Phi)(\Theta : \Xi)+2(\Psi : \Theta)(\Phi : \Xi) \right).
\end{equation}

Since \cref{eq:weak-Q} is nonlinear, we need to approximate the solution of its linearised version, i.e., find $\Theta\in \mathbf{H}^1_0(\Omega)$ such that
\begin{equation}
    \label{eq:linear-Q}
    \langle \mathcal{D}\mathcal{N}^n(\Qvec)\Theta,\Phi\rangle \coloneqq A^n(\Theta, \Phi) + 3B^n(\Qvec,\Qvec,\Theta,\Phi) + C^n(\Theta,\Phi) = - \mathcal{N}^n(\Qvec)\Phi \quad \forall \Phi\in \mathbf{H}^1_0(\Omega),
\end{equation}
where $\langle \cdot,\cdot\rangle$ represents the dual pairing between $\mathbf{H}^{-1}(\Omega)$ and $\mathbf{H}^1_0(\Omega)$.

Suppose $\Qvec_h\in \mathbf{V}_h$ approximates the solution of \cref{eq:weak-Q} with the conforming finite element method on a finite dimensional space
$\mathbf{V}_h\coloneqq \{\Pvec\in \mathbf{H}^1(\Omega): \Pvec\in \mathbb{Q}_{\deg}(T),\deg \ge 1, \forall T\in \mathcal{T}_h\}$.
Throughout this subsection we take $\deg\ge 1$.
Furthermore, we denote $\mathbf{V}_{h,0}\coloneqq \{\Pvec\in \mathbf{V}_h: \Pvec=\mathbf{0} \text{ on }\partial\Omega\}$ and $\mathbf{V}_{h,b}\coloneqq \{\Pvec\in \mathbf{V}_h: \Pvec=\Qvec_b \text{ on }\partial\Omega\}$.
We assume that the minimiser $\Qvec$ to be approximated is isolated, i.e., the linearised operator $\langle \mathcal{D}\mathcal{N}^n(\Qvec)\cdot,\cdot\rangle$ is nonsingular.
This
is equivalent to the following continuous inf-sup condition \cite[Eq. (2.8)]{maity-2020a-article}:
\begin{equation}
    \label{eq:inf-sup-DN}
    0<\beta_Q \coloneqq \adjustlimits\inf_{\underset{\|\Theta\|_1=1}{\Theta\in \mathbf{H}^1_0(\Omega)}}\sup_{\underset{\|\Phi\|_1=1}{\Phi\in \mathbf{H}^1_0(\Omega)}} \langle \mathcal{D}\mathcal{N}^n(\Qvec)\Theta,\Phi\rangle
    =  \adjustlimits\inf_{\underset{\|\Phi\|_1=1}{\Phi\in \mathbf{H}^1_0(\Omega)}} \sup_{\underset{\|\Theta\|_1=1}{\Theta\in \mathbf{H}^1_0(\Omega)}} \langle \mathcal{D}\mathcal{N}^n(\Qvec)\Theta,\Phi\rangle.
\end{equation}

With this inf-sup condition for $\langle\mathcal{D}\mathcal{N}^n(\Qvec)\cdot,\cdot\rangle$, we can obtain a stability result for the perturbed bilinear form $\langle\mathcal{D}\mathcal{N}^n(I_h \Qvec)\cdot,\cdot\rangle$ by following similar steps as in the proof of \cref{thm:weak-coer-per}. 
\begin{theorem}
    \label{thm:perturbed-Nn}
    (Stability of the perturbed bilinear form)
    Let $\Qvec$ be a regular isolated solution of the nonlinear continuous weak form \cref{eq:weak-Q} and $I_h \Qvec$ the interpolant of $\Qvec$.
    For a sufficiently small mesh size $h$, the following discrete inf-sup condition holds:
    \begin{equation}
        \label{eq:stability-perturb}
        0<\frac{\beta_Q}{2} 
        \leq \adjustlimits\inf_{\underset{\|\Theta\|_1=1}{\Theta\in \mathbf{H}^1_0(\Omega)}}\sup_{\underset{\|\Phi\|_1=1}{\Phi\in \mathbf{H}^1_0(\Omega)}}
        \langle \mathcal{D}\mathcal{N}^n(I_h \Qvec)\Theta,\Phi\rangle.
    \end{equation}
\end{theorem}

We give some auxiliary results about the operators $A^n(\cdot,\cdot)$, $B^n(\cdot,\cdot,\cdot,\cdot)$ and $C^n(\cdot,\cdot)$ that can be verified via the Cauchy--Schwarz inequality, the Poincar\'e inequality, and Sobolev embeddings.

\begin{lemma}\label{lem:bounded-An}
    (Boundedness and coercivity of $A^n(\cdot,\cdot)$)
    For $\Theta,\Phi\in \mathbf{H}^1_0(\Omega)$, there holds
    \begin{equation*}
        A^n(\Theta,\Phi) \lesssim \|\Theta\|_1\|\Phi\|_1 \text{ and }
        \|\Theta\|_1^2 \lesssim A^n(\Theta,\Theta).
    \end{equation*}
\end{lemma}

\begin{lemma}
    \label{lem:bounded-Bn}
    (Boundedness of $B^n(\cdot,\cdot,\cdot,\cdot)$, $C^n(\cdot,\cdot)$)
    For $\Psi,\Phi,\Theta,\Xi\in \mathbf{H}^1(\Omega)$, there holds
    \begin{equation}
        \label{eq:bound1-Bn}
        B^n(\Psi,\Phi,\Theta,\Xi) \lesssim \|\Psi\|_1 \|\Phi\|_1 \|\Theta\|_1 \|\Xi\|_1,\quad C^n(\Psi,\Phi) \lesssim \|\Psi\|_1 \|\Phi\|_1,
    \end{equation}
    and for $\Psi,\Phi\in \mathbf{H}^{\Bbbk}(\Omega)$, $\Bbbk\ge 2$, $\Theta,\Xi\in \mathbf{H}^1(\Omega)$,
    \begin{equation}
        \label{eq:bound2-Bn}
        B^n(\Psi,\Phi,\Theta,\Xi) \lesssim \|\Psi\|_\Bbbk \|\Phi\|_\Bbbk \|\Theta\|_1 \|\Xi\|_1.
    \end{equation}
\end{lemma}


To proceed to error estimates for the nonlinear problem \cref{eq:weak-Q}, we define the nonlinear map $\psi: \mathbf{V}_h\to \mathbf{V}_h$ by
\begin{equation*}
    \langle \mathcal{D}\mathcal{N}^n(I_h \Qvec) \psi(\Theta_h), \Phi_h\rangle = 3B^n(I_h \Qvec, I_h \Qvec, \Theta_h, \Phi_h) - B^n(\Theta_h,\Theta_h,\Theta_h,\Phi_h)
\end{equation*}
for $\Theta_h,\Phi_h\in \mathbf{V}_{h,0}$.
Due to the stability result of \cref{thm:perturbed-Nn}, the nonlinear map $\psi$ is well-defined.
We define the local ball $\mathcal{B}_\rho(I_h\Qvec)\coloneqq \{\Pvec_h\in \mathbf{V}_h:\|I_h\Qvec-\Pvec_h\|_1\le \rho\}$.
The following two auxiliary lemmas provide the necessary components for the application of Brouwer's fixed point theorem.

\begin{lemma}
    \label{lem:mappingball-Q}
    (Mapping from a ball to itself)
    Let $\Qvec$ be a regular isolated solution of the continuous nonlinear weak problem \cref{eq:weak-Q}.
    For a sufficiently small mesh size $h$, there exists a positive constant $r(h)>0$ such that:
    \begin{equation*}
        \|\Pvec_h - I_h \Qvec\|_1 \le r(h) \Rightarrow \|\psi(\Pvec_h)- I_h \Qvec\|_1 \le r(h)\quad \forall \Pvec_h\in \mathbf{V}_{h,0}.
    \end{equation*}
\end{lemma}
\begin{remark}
    In fact, the choice of $r(h)$ can be taken as $r(h)=\mathcal{O}(h^{\min\{\deg, \Bbbk_Q-1\}})$ in the proof of \cref{lem:mappingball-Q}.
    Here, $\Bbbk_Q\ge 2$ denotes the regularity index of $\Qvec$, i.e., $\Qvec\in \mathbf{H}^{\Bbbk_Q}(\Omega)$.
\end{remark}

\begin{lemma}
    \label{lem:contraction-Q}
    (Contraction result)
    For a sufficiently small mesh size $h$ and any $\Pvec_1,\Pvec_2\in \mathcal{B}_{r(h)}(I_h \Qvec)$, there holds
    \begin{equation}
    \|\psi(\Pvec_1) - \psi(\Pvec_2)\|_1 \lesssim h^{\min\{\deg,\Bbbk_Q-1\}} \|\Pvec_1-\Pvec_2\|_1.
    \end{equation}
\end{lemma}
\begin{remark}
    In the proof of \cref{lem:contraction-Q}, we have particularly used the stability property of the perturbed bilinear form as given by \cref{thm:perturbed-Nn}.
\end{remark}

Hence, the existence and local uniqueness of the discrete solution $\Qvec_h$ can be derived by following similar steps as in the proof of \cref{thm:main1u}.

\begin{theorem}
    \label{thm:H1-Q}
    (Convergence in $\|\cdot\|_1$-norm)
    Let $\Qvec$ be a regular isolated solution of the nonlinear problem \cref{eq:weak-Q}.
    For a sufficiently small $h$, there exists a unique solution $\Qvec_h$ of the discrete nonlinear problem \cref{eq:weak-Q} within the local ball $\mathcal{B}_{r(h)}(I_h \Qvec)$.
    Furthermore, we have
    $
        \|\Qvec - \Qvec_h\|_1 \lesssim h^{\min\{\deg,\Bbbk_Q-1\}}.
    $
\end{theorem}


We again employ an Aubin--Nitsche duality argument to derive $L^2$ error estimates.
To this end, we consider the following linear dual problem to the primal nonlinear problem \cref{eq:strong-Q}: find $\mathbf{N}\in \mathbf{H}^1_0(\Omega)$ such that
\begin{equation}
    \label{eq:dual}
    \begin{cases}
        -K\Delta \Nvec + 4l|\Qvec|^2 \Nvec +8l(\Qvec: \mathbf{N})\Qvec - 2l \Nvec = \Gvec & \text{in }\Omega ,\\
        \Nvec =\mathbf{0} & \text{on }\partial\Omega,
    \end{cases}
\end{equation}
for a given $\Gvec\in \mathbf{L}^2(\Omega)$ (we will make a particular choice for $\Gvec$ in the proof of \cref{thm:L2-Q}).
The weak form of \cref{eq:dual} is to find $\Nvec\in \mathbf{H}^1_0(\Omega)$ such that
\begin{equation}
    \label{eq:weak-dual-Q}
    \langle \mathcal{D}\mathcal{N}^n(\Qvec) \Nvec,\Phi\rangle = A^n(\Nvec, \Phi) + 3B^n(\Qvec,\Qvec,\Nvec,\Phi)+C^n(\Nvec,\Phi)= (\Gvec,\Phi)_0 \quad \forall \Phi\in \mathbf{H}^1_0(\Omega).
\end{equation}

To derive the $L^2$ a priori error estimates, we need two more auxiliary results.
\begin{lemma}
    \label{lem:H2-An}
    For $\Qvec\in \mathbf{H}^{\Bbbk_Q}(\Omega)\cap \mathbf{H}^1_b(\Omega)$, $\Bbbk_Q\ge 2$, $\Nvec \in \mathbf{H}^2(\Omega)\cap \mathbf{H}^1_0(\Omega)$ and $I_h \Qvec \in \mathbf{V}_h\subset \mathbf{H}^1_b(\Omega)$, it holds that
    \begin{equation*}
        A^n(I_h \Qvec-\Qvec, \Nvec) \lesssim h^{\min\{\deg+1, \Bbbk_Q\}} \|\Qvec\|_{\Bbbk_Q}\|\Nvec \|_2.
    \end{equation*}
\end{lemma}

\begin{lemma}
    (Boundedness of the dual solution in the $H^2$-norm)
    \label{lem:G-bound}
    The solution $\Nvec$ to the weak form \cref{eq:weak-dual-Q} of the dual linear problem belongs to $\mathbf{H}^2(\Omega)\cap \mathbf{H}^1_0(\Omega)$ and it holds that
    \begin{equation}
        \label{eq:bdd-M}
        \|\Nvec \|_2 \lesssim \|\Gvec \|_0.
    \end{equation}
\end{lemma}

Finally, we are ready to deduce an optimal $L^2$ error estimate.
\begin{theorem}
    \label{thm:L2-Q}
    ($L^2$ error estimate)
    Let $\Qvec$ be a regular solution of the nonlinear weak problem \cref{eq:weak-Q} and $\Qvec_h$ be the approximate solution to the discrete problem (having the same weak formulation as \cref{eq:weak-Q}). Then
    \begin{equation}
        \label{eq:l2-Q}
        \|\Qvec-\Qvec_h\|_0 \lesssim h^{\min\{\deg+1,\Bbbk_Q\}} \left(2+\left(3+h+h^2+h^{\min\{\deg,\Bbbk_Q-1\}}+h^{\min\{\deg+1,\Bbbk_Q\}}\right)\|\Qvec\|_{\Bbbk_Q}^2\right)\|\Qvec\|_{\Bbbk_Q}.
    \end{equation}
\end{theorem}

We will verify these results in the next section.

\section{Numerical experiments}
\label{sec:num}

The proceeding section presents some a priori error estimates for both $\Qvec$ and $u$ in the decoupled case $q=0$.
We now test the convergence rate of the finite element approximations by the method of manufactured solutions (MMS) and experimentally investigate the coupled case $q\ne 0$ in two dimensions.
To this end, we choose a nontrivial solution for each state variable and add an appropriate source term to the equilibrium equations, 
thus modifying the energy accordingly.
We can then compute the numerical convergence order.

\subsection{Test 1: on the unit square}

In this test, the numerical runs are performed on the unit square $\Omega=(0,1)^2$ and we take the following exact expressions for each state variable,
\begin{equation}
    \label{eq:mms}
    \begin{aligned}
        Q_{11}^{e} &= \left(\cos \left(\frac{\pi (2y-1)(2x-1)}{8}\right)\right)^2 - \frac{1}{2},\\
        Q_{12}^{e} &= \cos\left( \frac{\pi (2y-1)(2x-1)}{8}\right) \sin\left(\frac{\pi (2y-1)(2x-1)}{8}\right),\\
        u^{e} &= 10\left((x-1)x(y-1)y \right)^3.
    \end{aligned}
\end{equation}
Then, in conducting the MMS, we are to solve the following governing equations
\begin{equation*}
    \begin{cases}
         4Bq^4 u^2Q_{11} + 2Bq^2 u \left(\partial_x^2 u - \partial_y^2 u\right) - 2K\Delta Q_{11} - 4l Q_{11} + 16 lQ_{11}\left(Q_{11}^2+Q_{12}^2\right) = \mathfrak{s}_1, \\
         4Bq^4 u^2Q_{12} + 4Bq^2 u \left(\partial_x\partial_y u\right) - 2K\Delta Q_{12} - 4l Q_{12} + 16 lQ_{12}\left(Q_{11}^2+Q_{12}^2\right) = \mathfrak{s}_2, \\
         a_1 u + a_2 u^2 +a_3 u^3+ 2B\nabla\cdot(\nabla\cdot(\mathcal{D}^2 u)) + Bq^4 \left(4\left(Q_{11}^2+Q_{12}^2\right)+1\right)u + 2Bq^2(t_1+t_2) = \mathfrak{s}_3,
    \end{cases}
\end{equation*}
subject to Dirichlet boundary conditions for both $u$ and $\Qvec$ and a natural boundary condition for $u$.
Here, source terms $\mathfrak{s}_1$, $\mathfrak{s}_2$ and $\mathfrak{s}_3$ are derived by substituting \cref{eq:mms} to the left hand sides,
and $t_1$ and $t_2$ are given by
\begin{equation*}
    \begin{split}
        t_1 &\coloneqq (Q_{11}+1/2) \partial_x^2 u + (-Q_{11}+1/2)\partial_y^2 u + 2Q_{12}\partial_x\partial_y u, \\
        t_2 &\coloneqq \partial_x^2\left(u\left(Q_{11}+1/2\right)\right) + \partial_y^2(u(-Q_{11}+1/2)) + 2\partial_x\partial_y (u Q_{12}).
    \end{split}
\end{equation*}

We partition the domain $\Omega = (0, 1)^2$ into $N\times N$ squares with uniform mesh size $h=\frac{1}{N}$ ($N=6$, $12$, $24$, $48$) and denote the numerical solutions by $u_h$, $Q_{11,h}$ and $Q_{12,h}$.
The numerical errors of $u$ and $\Qvec$ in the $\|\cdot\|_0$-, $\|\cdot\|_1$- and $\vertiii{\cdot}_h$-norms are defined as
    \begin{align*}
        &\|\textbf{e}_u\|_0 = \|u^{e}-u_h\|_0,
        \quad \|\textbf{e}_u\|_1 = \|u^{e}-u_h\|_1, \quad \vertiii{\textbf{e}_u}_h = \vertiii{u^e-u_h}_h, \\
        \|\textbf{e}_\Qvec\|_0 = &\|(Q_{11}^{e}, Q_{12}^{e}) - (Q_{11,h},Q_{12,h})\|_0,
        \quad \|\textbf{e}_\Qvec\|_1 = \|(Q_{11}^{e}, Q_{12}^{e}) - (Q_{11,h},Q_{12,h})\|_1.
    \end{align*}
The convergence order is then calculated from the formula
$
    \log_2\left( \frac{\text{error}_{h/2}}{\text{error}_h}\right).
$
Throughout this section, we use the parameter values
$a_1 = -10$, $a_2 = 0$, $a_3 =10$, $B = 10^{-5}$, $K = 0.3$ and $l = 30$,
similar to the simulations of oily streaks in \cite{xia-2021-article}.

\begin{remark}
    Since this is purely a numerical verification exercise, the manufactured solution can be physically unrealistic.
    However, we must specify a reasonable initial guess for Newton's method, due to the nonlinearity of the problem.
    The initial guess throughout this section is taken to be $\left(\frac{1}{2}(\text{exact solution}) + 10^{-9}\right)$.
\end{remark}

\subsubsection{Convergence rate for $q=0$}
For $\Qvec$ we expect both optimal $H^1$ and $L^2$ rates, as illustrated in \cref{thm:L2-Q,thm:H1-Q}.
\cref{table:decoupled-Q} presents the numerical convergence rate for the finite elements $[\mathbb{Q}_1]^2$, $[\mathbb{Q}_2]^2$ and $[\mathbb{Q}_3]^2$.
Optimal $L^2$ and $H^1$ rates are shown with all choices of finite elements, as predicted.

\begin{table}[!ht]
\centering
    \begin{tabular}{ccllll}
        \toprule
        & $N=\frac{1}{h}$ & $\|\textbf{e}_\Qvec\|_0$ & rate & $\|\textbf{e}_\Qvec\|_1$ & rate \\
         \midrule
         \multirow{5}{*}{$[\mathbb{Q}_1]^2$}
        & 6 & 8.12 $\times 10^{-4}$ &-- & 3.78 $\times 10^{-2}$ &-- \\
        &12 & 2.02 $\times 10^{-4}$ & 2.01 & 1.88 $\times 10^{-2}$ & 1.01\\
        &24 & 5.05 $\times 10^{-5}$ & 2.00 & 9.39 $\times 10^{-3}$ & 1.00\\
        &48 & 1.26 $\times 10^{-5}$ & 2.00 & 4.69 $\times 10^{-3}$ & 1.00\\
        \bottomrule
         \multirow{5}{*}{$[\mathbb{Q}_2]^2$}
        & 6 & 2.92 $\times 10^{-5}$ &-- & 1.11 $\times 10^{-3}$ &-- \\
        &12 & 3.90 $\times 10^{-6}$ & 2.90 & 2.71 $\times 10^{-4}$ & 2.04\\
        &24 & 5.02 $\times 10^{-7}$ & 2.96 & 6.72 $\times 10^{-5}$ & 2.01\\
        &48 & 6.36 $\times 10^{-8}$ & 2.99 & 1.68 $\times 10^{-5}$ & 2.00\\
        \bottomrule
        \multirow{5}{*}{$[\mathbb{Q}_3]^2$}
        & 6 & 3.02 $\times 10^{-7}$ &-- & 2.25 $\times 10^{-5}$ &-- \\
        &12 & 2.17 $\times 10^{-8}$ & 3.80 & 2.72 $\times 10^{-6}$ & 3.05\\
        &24 & 1.45 $\times 10^{-9}$ & 3.90 & 3.34 $\times 10^{-7}$ & 3.03\\
        &48 & 9.33 $\times 10^{-11}$ & 3.96 & 4.13 $\times 10^{-8}$ & 3.01\\
        \bottomrule
    \end{tabular}
    \caption{Test 1: Convergence rates for $\Qvec$ with different degrees of polynomial approximation, in the decoupled case $q=0$.}
    \label{table:decoupled-Q}
\end{table}

Regarding the density variation $u$, we first present the convergence behaviour of the consistent discrete formulation \cref{eq:weak-u} with penalty parameter $\epsilon = 1$, since we have proven the optimal error rate in the mesh-dependent norm $\vertiii{\cdot}_h$.
The errors and convergence orders are listed in \cref{table:u-eps1}.
Optimal rates are observed in the $\vertiii{\cdot}_h$-norm.
Furthermore, optimal orders of convergence in the $L^2$-norm are shown for approximating polynomials of degree greater than $2$, while a sub-optimal rate in the $L^2$-norm is given for piecewise quadratic polynomials, exactly as expected.
Sub-optimal convergence rates for quadratic polynomials were also illustrated in the numerical results of \cite{suli-2007-article}.
We also tested the convergence with the penalty parameter $\epsilon=5\times 10^4$ and found that the discrete norms are very similar to \cref{table:u-eps1}.
We therefore avoid repeating the details here.

\begin{table}[!ht]
\centering
    \begin{tabular}{ccllllll}
        \toprule
        & $N=\frac{1}{h}$ & $\|\textbf{e}_u\|_0$ & rate & $\|\textbf{e}_u\|_1$ & rate& $\vertiii{\textbf{e}_u}_h$ & rate \\
         \midrule
         \multirow{4}{*}{$\mathbb{Q}_2$}
        & 6 & 1.17 $\times 10^{-5}$ & -- & 3.46 $\times 10^{-4}$ &-- & 1.36 $\times 10^{-2}$ & --\\
        & 12 & 2.60 $\times 10^{-6}$ & 2.17 & 9.81 $\times 10^{-5}$ & 1.82 & 7.25 $\times 10^{-3}$ & 0.91\\
        & 24 & 6.37 $\times 10^{-7}$ & 2.03 & 2.54 $\times 10^{-5}$ & 1.95 & 3.54 $\times 10^{-3}$ & 1.03\\
        & 48 & 1.82 $\times 10^{-7}$ & 1.80 & 6.88 $\times 10^{-6}$ & 1.88 & 1.76 $\times 10^{-3}$ & 1.01\\
        \bottomrule
        \multirow{4}{*}{$\mathbb{Q}_3$}
        &6 & 4.73 $\times 10^{-6}$ & -- & 1.32 $\times 10^{-4}$ &-- & 4.98 $\times 10^{-3}$ & --\\
        &12 & 3.32 $\times 10^{-7}$ & 3.83 & 1.41 $\times 10^{-5}$ & 3.23 & 9.96 $\times 10^{-4}$ & 2.32\\
        &24 & 2.12 $\times 10^{-8}$ & 3.97 & 1.63 $\times 10^{-6}$ & 3.12 & 2.46 $\times 10^{-4}$ & 2.02\\
        &48 & 1.32 $\times 10^{-9}$ & 4.00 & 1.99 $\times 10^{-7}$ & 3.03 & 6.14 $\times 10^{-5}$ & 2.00\\
        \bottomrule
        \multirow{4}{*}{$\mathbb{Q}_4$}
        &6 & 2.01 $\times 10^{-7}$ & -- & 7.76 $\times 10^{-6}$ &-- & 3.94 $\times 10^{-4}$ & --\\
        &12 & 5.40 $\times 10^{-9}$ & 5.22 & 4.30 $\times 10^{-7}$ & 4.17 & 4.88 $\times 10^{-5}$ & 3.01\\
        &24 & 1.68 $\times 10^{-10}$ & 5.00 & 2.68 $\times 10^{-8}$ & 4.00 & 6.11 $\times 10^{-6}$ & 2.99\\
        &48 & 5.27 $\times 10^{-12}$ & 4.99 & 1.68 $\times 10^{-9}$ & 3.99 & 7.64 $\times 10^{-7}$ & 3.00\\
        \bottomrule
    \end{tabular}
    \caption{Test 1: Convergence rates using the consistent discrete formulation \cref{eq:weak-u} with penalty parameter $\epsilon=1$ and different polynomial degrees, in the decoupled case $q=0$.}
    \label{table:u-eps1}
\end{table}

We next give the error rates for the inconsistent discrete formulation \cref{eq:weak-u-actual}.
We illustrate the discrete norms and the computed convergence rates in \cref{table:uactual-eps1} with the penalty parameter $\epsilon=1$.
It can be observed that only first order convergence is obtained in the $H^2$-like norm $\vertiii{\cdot}_h$ even with different approximating polynomials.
Moreover, we notice by comparing \cref{table:u-eps1,table:uactual-eps1} that the convergence rate deteriorates slightly for polynomials of degree 3 (although not for degree 4).
This, however, can be improved by choosing a larger penalty parameter, as shown in \cref{table:uactual-eps5e4} with $\epsilon=5\times 10^4$, where optimal rates are shown for the discrete norms $\vertiii{\cdot}_h$, $\|\cdot\|_1$ and $\|\cdot\|_0$ for all polynomial degrees (except only sub-optimal in $\|\cdot\|_0$ when a piecewise quadratic polynomial is used as the approximation). The inconsistent discrete formulation appears to be a reasonable choice when a sufficiently large penalty parameter is used.

\begin{table}[!ht]
\centering
    \begin{tabular}{ccllllll}
        \toprule
        & $N=\frac{1}{h}$ & $\|\textbf{e}_u\|_0$ & rate & $\|\textbf{e}_u\|_1$ & rate& $\vertiii{\textbf{e}_u}_h$ & rate \\
         \midrule
         \multirow{4}{*}{$\mathbb{Q}_2$}
        & 6 & 3.50 $\times 10^{-6}$ & -- & 1.06 $\times 10^{-4}$ &-- & 5.60 $\times 10^{-3}$ & --\\
        & 12 & 8.76 $\times 10^{-8}$ & 5.32 & 5.41 $\times 10^{-6}$ & 4.29 & 2.56 $\times 10^{-3}$ & 1.13\\
        & 24 & 1.77 $\times 10^{-8}$ & 2.31 & 7.47 $\times 10^{-7}$ & 2.86 & 1.28 $\times 10^{-3}$ & 0.99\\
        & 48 & 4.35 $\times 10^{-9}$ & 2.02 & 1.24 $\times 10^{-7}$ & 2.56 & 6.42 $\times 10^{-4}$ & 1.00\\
        \bottomrule
        \multirow{4}{*}{$\mathbb{Q}_3$}
        &6 & 6.47 $\times 10^{-6}$ & -- & 1.86 $\times 10^{-4}$ &-- & 7.59 $\times 10^{-3}$ & --\\
        &12 & 3.40 $\times 10^{-7}$ & 4.25 & 1.73 $\times 10^{-5}$ & 3.43 & 2.74 $\times 10^{-3}$ & 1.47\\
        &24 & 1.98 $\times 10^{-8}$ & 4.10 & 2.03 $\times 10^{-6}$ & 3.09 & 1.31 $\times 10^{-3}$ & 1.07\\
        &48 & 3.73 $\times 10^{-9}$ & 2.39 & 2.63 $\times 10^{-7}$ & 2.95 & 6.45 $\times 10^{-4}$ & 1.02\\
        \bottomrule
        \multirow{4}{*}{$\mathbb{Q}_4$}
        &6 & 2.05 $\times 10^{-7}$ & -- & 7.85 $\times 10^{-6}$ &-- & 3.93 $\times 10^{-4}$ & --\\
        &12 & 5.40 $\times 10^{-9}$ & 5.24 & 4.31 $\times 10^{-7}$ & 4.19 & 4.88 $\times 10^{-5}$ & 3.01\\
        &24 & 1.68 $\times 10^{-10}$ & 5.00 & 2.68 $\times 10^{-8}$ & 4.01 & 6.11 $\times 10^{-6}$ & 3.00\\
        &48 & 5.27 $\times 10^{-12}$ & 5.00 & 1.67 $\times 10^{-9}$ & 4.00 & 7.64 $\times 10^{-7}$ & 3.00\\
        \bottomrule
    \end{tabular}
    \caption{Test 1: Convergence rates using the inconsistent discrete formulation \cref{eq:weak-u-actual} with penalty parameter $\epsilon=1$ and different polynomial degrees, in the decoupled case $q=0$.}
    \label{table:uactual-eps1}
\end{table}

\begin{table}[!ht]
\centering
    \begin{tabular}{ccllllll}
        \toprule
        & $N=\frac{1}{h}$ & $\|\textbf{e}_u\|_0$ & rate & $\|\textbf{e}_u\|_1$ & rate& $\vertiii{\textbf{e}_u}_h$ & rate \\
         \midrule
         \multirow{4}{*}{$\mathbb{Q}_2$}
        & 6 & 1.17 $\times 10^{-5}$ & -- & 3.48 $\times 10^{-4}$ &-- & 1.36 $\times 10^{-2}$ & --\\
        & 12 & 2.62 $\times 10^{-6}$ & 2.16 & 9.86 $\times 10^{-5}$ & 1.82 & 7.26 $\times 10^{-3}$ & 0.91\\
        & 24 & 6.38 $\times 10^{-7}$ & 2.04 & 2.54 $\times 10^{-5}$ & 1.96 & 3.54 $\times 10^{-3}$ & 1.03\\
        & 48 & 1.82 $\times 10^{-7}$ & 1.81 & 6.88 $\times 10^{-6}$ & 1.88 & 1.76 $\times 10^{-3}$ & 1.01\\
        \bottomrule
        \multirow{4}{*}{$\mathbb{Q}_3$}
        &6 & 4.80 $\times 10^{-6}$ & -- & 1.35 $\times 10^{-4}$ &-- & 4.92 $\times 10^{-3}$ & --\\
        &12 & 3.35 $\times 10^{-7}$ & 3.84 & 1.43 $\times 10^{-5}$ & 3.23 & 9.86 $\times 10^{-4}$ & 2.32\\
        &24 & 2.14 $\times 10^{-8}$ & 3.97 & 1.63 $\times 10^{-6}$ & 3.13 & 2.45 $\times 10^{-4}$ & 2.01\\
        &48 & 1.33 $\times 10^{-9}$ & 4.01 & 1.99 $\times 10^{-7}$ & 3.04 & 6.13 $\times 10^{-5}$ & 2.00\\
        \bottomrule
        \multirow{4}{*}{$\mathbb{Q}_4$}
        &6 & 2.05 $\times 10^{-7}$ & -- & 7.85 $\times 10^{-6}$ &-- & 3.93 $\times 10^{-4}$ & --\\
        &12 & 5.40 $\times 10^{-9}$ & 5.24 & 4.31 $\times 10^{-7}$ & 4.19 & 4.88 $\times 10^{-5}$ & 3.01\\
        &24 & 1.68 $\times 10^{-10}$ & 5.00 & 2.68 $\times 10^{-8}$ & 4.01 & 6.11 $\times 10^{-6}$ & 3.00\\
        &48 & 5.27 $\times 10^{-12}$ & 5.00 & 1.67 $\times 10^{-9}$ & 4.00 & 7.64 $\times 10^{-7}$ & 3.00\\
        \bottomrule
    \end{tabular}
    \caption{Test 1: Convergence rates using the inconsistent discrete formulation \cref{eq:weak-u-actual} with penalty parameter $\epsilon=5\times 10^4$ and different polynomial degrees, in the decoupled case $q=0$.}
    \label{table:uactual-eps5e4}
\end{table}

\subsubsection{Convergence rate for $q\ne 0$}
We next investigate the numerical convergence behaviour in the coupled case, i.e., $q\ne 0$, in this subsection. Its analysis remains future work,
but since it is the coupled case that is solved in practice it is important to assure ourselves that the discretisation is sensible.
For brevity, we fix the model parameter $q=30$. 

We examine the inconsistent discretisation for $u$ with penalty parameter $\epsilon=5\times 10^4$.
In unreported preliminary experiments, we observed that the error in $\Qvec$ is governed by the lower of the degrees of the polynomials used for $\Qvec$ and $u$.
We thus give the convergence rates for $u$ and $\Qvec$ separately in \cref{table:q30-eps5e4-u,table:q30-eps5e4-Q}, with the other degree fixed appropriately.
It can be seen that $\Qvec$ retains optimal rates in both the $H^1$ and $L^2$ norms, and though there are some fluctuations of the order for $u$, it still possesses very similar convergence rates when compared with the decoupled case described in \cref{table:uactual-eps5e4}.

\begin{table}[!ht]
\centering
    \begin{tabular}{ccllllll}
        \toprule
        &$N=\frac{1}{h}$ & $\|\textbf{e}_u\|_0$ & rate & $\|\textbf{e}_u\|_1$ & rate & $\vertiii{\textbf{e}_u}_h$ & rate \\
         \midrule
         \multirow{4}{*}{$\mathbb{Q}_2$}
        &6 & 1.21 $\times 10^{-5}$ & -- & 3.59 $\times 10^{-4}$ &-- & 1.37 $\times 10^{-2}$ & -- \\
        &12 & 3.98 $\times 10^{-6}$ & 1.61 & 1.42 $\times 10^{-4}$ & 1.34 & 8.30 $\times 10^{-3}$ & 0.72 \\
        &24 & 1.57 $\times 10^{-6}$ & 1.35 & 4.99 $\times 10^{-5}$ & 1.51 & 3.89 $\times 10^{-3}$ & 1.09 \\
        &48 & 2.58 $\times 10^{-7}$ & 2.60 & 9.06 $\times 10^{-6}$ & 2.46 & 1.78 $\times 10^{-3}$ & 1.13 \\
        \bottomrule
        \multirow{4}{*}{$\mathbb{Q}_3$}
        &6 & 7.36 $\times 10^{-6}$ & -- & 2.25 $\times 10^{-4}$ &-- & 9.10 $\times 10^{-3}$ &-- \\
        &12 & 4.13 $\times 10^{-7}$ & 4.16 & 1.86 $\times 10^{-5}$ & 3.60 & 1.11 $\times 10^{-3}$ & 3.03 \\
        &24 & 4.23 $\times 10^{-8}$ & 3.29 & 2.24 $\times 10^{-6}$ & 3.05 & 2.53 $\times 10^{-4}$ & 2.14 \\
        &48 & 3.01 $\times 10^{-9}$ & 3.81 & 2.28 $\times 10^{-7}$ & 3.29 & 6.15 $\times 10^{-5}$ & 2.04 \\
        \bottomrule
    \end{tabular}
    \caption{Test 1: Convergence rates for $u$ with $q=30$ and penalty parameter $\epsilon=5\times 10^4$ with the inconsistent discretisation \cref{eq:weak-u-actual} for $u$, fixing the approximation for $\Qvec$ to be with the $[\mathbb{Q}_2]^2$ element.}
\label{table:q30-eps5e4-u}
\end{table}

\begin{table}[!ht]
\centering
    \begin{tabular}{ccllll}
        \toprule
        &$N=\frac{1}{h}$ & $\|\textbf{e}_\Qvec\|_0$ & rate & $\|\textbf{e}_\Qvec\|_1$ & rate \\
         \midrule
         \multirow{4}{*}{$[\mathbb{Q}_1]^2$}
        &6 & 8.12 $\times 10^{-4}$ &-- & 3.78 $\times 10^{-2}$ & --\\
        &12 &2.02 $\times 10^{-4}$ & 2.01 & 1.88 $\times 10^{-2}$ & 1.01 \\
        &24 &5.05 $\times 10^{-5}$ & 2.00 & 9.39 $\times 10^{-3}$ & 1.00 \\
        &48 &1.26 $\times 10^{-5}$ & 2.00 & 4.69 $\times 10^{-3}$ & 1.00 \\
        \bottomrule
        \multirow{4}{*}{$[\mathbb{Q}_2]^2$}
        &6 & 2.92 $\times 10^{-5}$ &-- & 1.11 $\times 10^{-3}$ & --\\
        &12 &3.90 $\times 10^{-6}$ & 2.90 & 2.71 $\times 10^{-4}$ &2.04 \\
        &24 &5.02 $\times 10^{-7}$ & 2.96 & 6.72 $\times 10^{-5}$ & 2.01 \\
        &48 &6.37 $\times 10^{-8}$ & 2.98 & 1.68 $\times 10^{-5}$ & 2.00 \\
        \bottomrule
        \multirow{4}{*}{$[\mathbb{Q}_3]^2$}
        &6 & 3.02 $\times 10^{-7}$ & -- & 2.25 $\times 10^{-5}$ &-- \\
        &12 & 2.17 $\times 10^{-8}$ & 3.80 & 2.72 $\times 10^{-6}$ & 3.05\\
        &24 & 1.45 $\times 10^{-9}$ & 3.90 & 3.34 $\times 10^{-7}$ & 3.03\\
        &48 & 9.32 $\times 10^{-11}$ & 3.96 & 4.13 $\times 10^{-08}$ & 3.01\\
        \bottomrule
    \end{tabular}
    \caption{Test 1: Convergence rates for $\Qvec$ with $q=30$ and penalty parameter $\epsilon=5\times 10^4$ with the inconsistent discretisation \cref{eq:weak-u-actual} for $u$, fixing the approximation for $u$ to be with the $\mathbb{Q}_3$ element.}
    \label{table:q30-eps5e4-Q}
\end{table}

\begin{remark}
    We also tested the convergence with the consistent weak formulation for $u$ under the same numerical settings as in \cref{table:q30-eps5e4-u,table:q30-eps5e4-Q}.
    We found that in both cases they present very similar convergence behaviour and thus we omit the details here.
\end{remark}

\subsection{Test 2: on the unit disc}

For the second set of experiments, we provide numerical results for an exact solution $u^e$ with only $H^3$-regularity, instead of $C^\infty$ as in \cref{eq:mms}. Our goal is to investigate whether the $H^4$-regularity assumption on $u$ can be relaxed. To this end, we consider a triangular mesh of $\Omega=\{(x,y)\ |\ x^2+y^2 < 1\}$ and choose $u^e$ to be
\begin{equation}
    \label{eq:test2-ue}
    u^e=(x^2+y^2)^{3/2},
\end{equation}
and choose the same exact solution for $Q^e_{11}$ and $Q^e_{12}$ as in \cref{eq:mms}. The exact solution given by~\cref{eq:test2-ue} is in $H^3(\Omega)$ but not in $H^4(\Omega)$, hence violating the regularity assumption of the analysis in \cref{sec:apriori-u}.

The resulting convergence rates are reported in tables~\ref{table:test2-Q} and~\ref{table:test2-uactual}.
\Cref{table:test2-Q} shows that optimal $H^1$ and $L^2$ rates are achieved for $\Qvec$ with three different choices of finite elements $[\mathbb{P}_1]^2, [\mathbb{P}_2]^2, [\mathbb{P}_3]^2$.
\Cref{table:test2-uactual} shows the convergence behaviour with penalty parameter $\varepsilon=1$ when using the inconsistent discrete formulation~\cref{eq:weak-u-actual}. In contrast to~\cref{table:uactual-eps1}, only first order convergence is obtained for the discrete norm $\vertiii{\cdot}_h$ and second-order convergence for $\|\cdot\|_0$ and $\|\cdot\|_1$, with both $\mathbb{P}_3$ and $\mathbb{P}_4$. Interestingly, \cref{table:test2-uactual} indicates no convergence when using $\mathbb{P}_2$ elements. It appears that the assumption $u \in H^4(\Omega)$ is necessary for our analysis, and that a different analysis should be carried out when this assumption no longer holds.

\begin{table}[!ht]
\centering
    \begin{tabular}{clllll}
        \toprule
        & No.\ of triangles & $\|\textbf{e}_\Qvec\|_0$ & rate & $\|\textbf{e}_\Qvec\|_1$ & rate \\
         \midrule
         \multirow{5}{*}{$[\mathbb{P}_1]^2$}
        & 60 & 6.08 $\times 10^{-2}$ &-- & 1.09 &-- \\
        &240 & 1.56 $\times 10^{-2}$ & 1.96 & 5.80 $\times 10^{-1}$ & 0.91\\
        &960 & 3.92 $\times 10^{-3}$ & 2.00 & 2.93 $\times 10^{-1}$ & 0.99\\
        &3840 & 9.83 $\times 10^{-4}$ & 2.00 & 1.47 $\times 10^{-1}$ & 1.00\\
        &15360 & 2.47 $\times 10^{-4}$  & 1.99 & 7.34 $\times 10^{-2}$ & 1.00\\
        \bottomrule
         \multirow{5}{*}{$[\mathbb{P}_2]^2$}
        & 60 & 8.97 $\times 10^{-3}$ &-- & 2.11 $\times 10^{-1}$ &-- \\
        &240 & 1.51 $\times 10^{-3}$ & 2.57 & 5.87 $\times 10^{-2}$ & 1.84\\
        &960 & 2.22 $\times 10^{-4}$ & 2.77 & 1.52 $\times 10^{-2}$ & 1.95\\
        &3840 & 3.02 $\times 10^{-5}$ & 2.88 & 3.85 $\times 10^{-3}$ & 1.98\\
        &15360 & 3.93 $\times 10^{-6}$  & 2.94 & 9.67 $\times 10^{-4}$ & 1.99\\
        \bottomrule
        \multirow{5}{*}{$[\mathbb{P}_3]^2$}
        & 60 & 1.08 $\times 10^{-3}$ &-- & 3.21 $\times 10^{-2}$ &-- \\
        &240 & 8.21 $\times 10^{-5}$ & 3.72 & 4.58 $\times 10^{-3}$ & 2.81 \\
        &960 & 5.52 $\times 10^{-6}$ & 3.89 & 5.92 $\times 10^{-4}$ & 2.95\\
        &3840 & 3.54 $\times 10^{-7}$ & 3.96 & 7.44 $\times 10^{-5}$ & 2.99\\
        &15360 & 2.23 $\times 10^{-8}$  & 3.99 & 9.31 $\times 10^{-6}$ & 3.00\\
        \bottomrule
    \end{tabular}
    \caption{Test 2: Convergence rates for $\Qvec$ with different degrees of polynomial approximation, in the decoupled case $q=0$.}
    \label{table:test2-Q}
\end{table}

\begin{table}[!ht]
\centering
    \begin{tabular}{clllllll}
        \toprule
        & No.\ of triangles & $\|\textbf{e}_u\|_0$ & rate & $\|\textbf{e}_u\|_1$ & rate& $\vertiii{\textbf{e}_u}_h$ & rate \\
         \midrule
         \multirow{5}{*}{$\mathbb{P}_2$}
        & 60 & 1.36 $\times 10^{-2}$ &-- & 2.77 $\times 10^{-1}$ &-- & 7.39 & --\\
        &240 & 1.58 $\times 10^{-3}$ & 3.11 & 3.86 $\times 10^{-2}$ & 2.84 & 2.39 & 1.63\\
        &960 & 7.76 $\times 10^{-4}$ & 1.02 & 1.57 $\times 10^{-2}$ & 1.30 & 1.44 & 0.73\\
        &3840 & 1.84 $\times 10^{-3}$ & -1.25 & 3.42 $\times 10^{-2}$ & -1.12 & 1.26 & 0.19\\
        &15360 & 2.77 $\times 10^{-3}$  & -0.59 & 5.29 $\times 10^{-2}$ & -0.63 & 1.60 & -0.34\\
        \bottomrule
        \multirow{5}{*}{$\mathbb{P}_3$}
        & 60 & 9.94 $\times 10^{-3}$ &-- & 2.21 $\times 10^{-1}$ &-- & 6.15 & --\\
        &240 & 3.99 $\times 10^{-3}$ & 1.32 & 8.49 $\times 10^{-2}$ & 1.38 & 3.03 & 1.02\\
        &960 & 1.33 $\times 10^{-4}$ & 4.90 & 4.57 $\times 10^{-3}$ & 4.22 & 1.27 & 1.26\\
        &3840 & 3.21 $\times 10^{-5}$ & 2.06 & 8.47 $\times 10^{-4}$ & 2.43 & 0.66 & 0.93\\
        &15360 & 9.22 $\times 10^{-6}$  & 1.80 & 2.11 $\times 10^{-4}$ & 2.00 & 0.34 & 0.97\\
        \bottomrule
        \multirow{5}{*}{$\mathbb{P}_4$}
        & 60 & 7.17 $\times 10^{-3}$ &-- & 1.65 $\times 10^{-1}$ &-- & 4.70 & --\\
        &240 & 1.34 $\times 10^{-3}$ & 2.42 & 3.84 $\times 10^{-2}$ & 2.10 & 2.39 & 0.98\\
        &960 & 1.18 $\times 10^{-4}$ & 3.50 & 4.54 $\times 10^{-3}$ & 3.08 & 1.28 & 0.90\\
        &3840 & 2.98 $\times 10^{-5}$ & 1.99 & 8.14 $\times 10^{-4}$ & 2.48 & 0.67 & 0.94\\
        &15360 & 8.15 $\times 10^{-6}$  & 1.87 & 1.86 $\times 10^{-4}$ & 2.13 & 0.34& 0.97\\
        \bottomrule
    \end{tabular}
    \caption{Test 2: Convergence rates using the inconsistent discrete formulation \cref{eq:weak-u-actual} with penalty parameter $\epsilon=1$ and different polynomial degrees, in the decoupled case $q=0$.}
    \label{table:test2-uactual}
\end{table}

\bibliographystyle{m2anplain}
\bibliography{main}
\end{document}